\documentclass[a4paper,12pt]{amsart}
\usepackage[utf8]{inputenc}
\usepackage[T1]{fontenc}
\usepackage[UKenglish]{babel}
\usepackage[margin=18mm]{geometry}
\usepackage{color}%cyan,red;magenta,green;yellow,blue

\usepackage{graphicx}

\usepackage{amsmath,amssymb,amsfonts,amsthm}
\usepackage{mathrsfs,eucal,dsfont}
\usepackage{verbatim,enumitem}

\usepackage{hyperref,url}

\newcommand{\R}{\mathds R}

\newcommand{\I}{\mathds 1}

\def\aa{\alpha}
\def\dd{\delta}
\def\d{{\rm d}}
\def\<{\langle}
\def\>{\rangle}

 \def\tt{\tilde}
 \def\ff{\frac}
 \def\ss{\sqrt}
\def\bb{\beta}

\def\R{\mathbb R}  \def\ff{\frac} \def\ss{\sqrt} 
 \def\kk{\kappa} 
\def\dd{\delta} \def\DD{\Delta} \def\vv{\varepsilon} 
\def\<{\langle} \def\>{\rangle}  \def\gg{\gamma}
  \def\nn{\nabla}  
\def\d{\text{\rm{d}}} \def\bb{\beta} \def\aa{\alpha} 
  \def\si{\sigma} 
 \def\beq{\begin{equation}}  
 
\def\e{\text{\rm{e}}}    
 \def\tt{\tilde} 
 \def\P{\mathbb P}

  \def\ll{\lambda}
 
\def\E{\mathbb E}

\def\to{\rightarrow}
\def\8{\infty}\def\3{\triangle}
\def\1{\lesssim}

\renewcommand{\bar}{\overline}
\renewcommand{\hat}{\widehat}
\renewcommand{\tilde}{\widetilde}

\newtheorem{theorem}{Theorem}[section]
\newtheorem{lemma}[theorem]{Lemma}

\newtheorem{corollary}[theorem]{Corollary}

\theoremstyle{definition}

\newtheorem{remark}[theorem]{Remark}

\numberwithin{equation}{section}
\begin{document}
\allowdisplaybreaks

\title[Non-asymptotic convergence bounds of modified EM  schemes] {Non-asymptotic convergence bounds of modified EM  schemes for non-dissipative SDEs}

\author{
Jianhai Bao\qquad
Jiaqing Hao\qquad Panpan Ren}
\date{}
\thanks{\emph{J.\ Bao:} Center for Applied Mathematics, Tianjin University, 300072  Tianjin,  China. \url{jianhaibao@tju.edu.cn}}

\thanks{\emph{J.\ Hao:}
Center for Applied Mathematics, Tianjin University, 300072  Tianjin,  China. \url{hjq_0227@tju.edu.cn}}

\thanks{\emph{P.\ Ren:}
 Mathematics Department, City University of Hong Kong, Tat Chee Av., Hong Kong,  China. \url{panparen@cityu.edu.hk}}

\maketitle

\begin{abstract}
In this paper, we address the issue on non-asymptotic convergence bounds of Euler-type schemes associated with
non-dissipative SDEs. On the one hand,  for non-degenerate SDEs with super-linear drifts, we
propose a novel modified Euler scheme  and establish the corresponding non-asymptotic convergence bound under the multiplicative type  quasi-Wasserstein distance
  by the aid of the asymptotic reflection by coupling. As a direct application of the theory derived, we
explore  the non-asymptotic convergence bound of the modified tamed/truncated Euler scheme
  and, as a byproduct, furnish  the associated non-asymptotic convergence rate under the $L^1$-Wasserstein distance although the
 dissipativity at infinity  is not in force. On the other hand, we tackle the non-asymptotic convergence analysis
of the Euler scheme corresponding to a kind of degenerate SDEs, where the underdamped Langevin SDE is a typical candidate. To handle such setting, we also appeal to a carefully tailored coupling approach, where the ingredient   in the coupling construction  lies in   that a proper   metric and a suitable substitute     in the cut-off function and the reflection matrix need to be chosen appropriately.  In addition, as a consequent  application, the non-asymptotic convergence bound and the $L^1$-Wasserstein convergence rate are revealed  for the  kinetic Langevin sampler.

\medskip

\smallskip

\end{abstract}

\section{Background and main results}
In the past few decades, the long-time asymptotics of Euler type schemes associated with globally dissipative SDEs has been established very well. For instance,     \cite{FG} explored  the uniform-in-time  error bound (in the $L^p$-moment sense) of the adapted Euler-Maruyama (EM for short) algorithm corresponding to a kind of  SDEs with non-globally Lipschitz drifts; \cite{BDMS} investigated the non-asymptotic $L^2$-Wasserstein bound   of the tamed unadjusted Langevin algorithm; \cite{CDHS} derived the strong $LLN$
and $CLT$ for the backward EM scheme concerned with  dissipative SDEs, functional SDEs, as well as stochastic evolution equations.  To address the issues under consideration in \cite{BDMS,CDHS,FG}, the approach based on the   synchronous  coupling
plays a crucial role by recalling that  the globally dissipative hypothesis is enforced therein. Whereas, the method on account of the synchronous coupling is not workable once the globally dissipative condition is weaken into the partially dissipative version.

With the remarkable progress made on  ergodicity  of SDEs with partially dissipative drifts (see, for instance,    \cite{Eberle,LMW,LMP,LW19}, to name just a few), the associated weak contraction  of the  EM scheme  has been investigated in depth in various scenarios when the drift of the underlying SDE   satisfies the global Lipschitz condition \cite{HJK}.
In particular, by the aid of the coupling approach,
the $L^1$-Wasserstein contraction of the EM algorithm was
studied in  \cite{EM} and  \cite{HMW} when the driven noise of the
non-degenerate SDE   under investigation  is the Brownian motion and the rotationally invariant $\alpha$-stable process, respectively. Via the   synchronous coupling and by constructing a quasi-metric function, which is comparable  to the cost function inducing the $L^2$-Wasserstein distance, the weak $L^2$-Wasserstein contraction of the EM scheme was treated in  \cite{LMP}. Meanwhile, the study  on the $L^1$-Wasserstein contraction of the EM method related to degenerate SDEs has also received much attention; see e.g. \cite{SW} concerning  three kinetic Langevin samplers. Additionally, we allude to \cite{GWXY} with regard to   a
quantitative error estimate   between SGD with momentum and underdamped Langevin diffusion.
Overall, regardless of  \cite{EM}, \cite{GWXY}, \cite{HMW}, \cite{LMP} or \cite{SW}, the drifts of underlying SDEs
are confined to be of linear growth \cite{HJK} so numerous interesting and important  SDEs with super-linear growth are excluded.

More recently, there also is progress on
the ergodicity of Euler type schemes associated with SDEs with sup-linear and partially dissipative drifts. Especially, a variant, which incorporates  the projected EM algorithm and the tamed EM scheme as two typical candidates,
of the classical EM scheme was proposed in \cite{BMW}. For this newly designed  scheme,
the   weak contraction under the $L^1$-Wasserstein distance, the weighted total variation distance, as well as the mixed probability distance between the $L^1$-Wasserstein distance and the total variation distance was studied in some detail by
constructing   a refined basic coupling or designing  an asymptotic coupling by reflection \cite{Wang15}. As for the projected EM scheme and the truncated EM algorithm,
the $L^2$-Wasserstein contraction was tackled in \cite{BH3} by making good use of the   synchronous coupling
and the associated contraction under a multiplicative type quasi-Wasserstein distance built  in a subtle way.
In addition, \cite{NMZ} probed into the non-asymptotic convergence bounds for modified tamed unadjusted Langevin algorithm in a non-convex setting, which improved greatly the corresponding convergence rate derived in \cite{DAS}. It is also worthy to mention the reference \cite{IS}, where, concerning a novel taming Langevin-based scheme,
   the non-asymptotic convergence bounds in the relative entropy was obtained 
by means of the Log-Sobolev inequality.

Regarding  fully non-dissipative   SDEs (i.e., they satisfy a Lyapunov condition in lieu of the  partially dissipative condition, which in the literature is also termed as the dissipativity at infinity),
the recent advances upon  the ergodicity of (functional) SDEs is also achieved.
Below, we  give  an overview on the related references. By initiating the so-called weak Harris' theorem and taking advantage of the generalized coupling approach,   the exponential contraction under a multiplicative type quasi-Wasserstein distance was explored  in \cite{HMS}, where the diffusion term of the underlying functional SDE is path-dependent. Later on, the pioneer work \cite{HMS} was extended
in \cite{BWY2} to SDEs with infinite memory by employing a concise Wasserstein coupling. Furthermore,  by invoking the mixed coupling (e.g., the mixture of the   synchronous coupling and the reflection coupling or the refined basic coupling) instead of the generalized coupling,
the exponential contraction under the quasi-Wasserstein distance has been  investigated in a more in depth fashion for non-degenerate/degenerate SDEs. More specifically,  we would like to refer to \cite{EGZ19}  for Langevin dynamics, \cite{EGA,Wang23} concerning McKean-Vlasov SDEs with the Brownian motion noise, and
 \cite{LMW} with regarding to  McKean-Vlasov SDEs, where the  driven noise is  a pure jump
L\'evy process.

As we mentioned previously, there has been a lot of effort in investigating the non-asymptotic convergence analysis for various  sampling algorithms
when the gradient of  the underlying potential is partially dissipative; see, for instance,  \cite{BH2,BH3,BDMS,GWXY,DAS,LSZ,NMZ} and references within.
 Nevertheless, the same issue is rather incomplete  when the associated partially dissipative condition is
 replaced by the fully non-dissipative counterpart (i.e., the Lyapunov type condition). So, in the context of the rapid advancement   on ergodicity of fully non-dissipative SDEs,
 it is quite natural to ask the following question: as regards the EM scheme and its variants associated with fully non-dissipative SDEs, how can we quantify  the non-asymptotic error bound under a suitable probability (quasi-)distance, and provide solid theoretical support for samplings of high target distributions  when  the partially dissipative condition (which has been presumed frequently in the literature) is no longer satisfied.
 The aforementioned issues impel us to pursue the  topic
   in the present work. In  addition, our ongoing paper is also  inspired partly  by  the work \cite{Suzuki}, in which a bound  on the EM-discretization error for Langevin dynamics with the {\it quadratic} potential was provided  with the help of an  approximate reflection coupling.

\subsection{Non-asymptotic convergence bound: non-degenerate case}\label{ssec1}
In this subsection, we focus on the following SDE:
\begin{align}\label{eq1}
\d X_t= b (X_t)  \,\d t+ \d W_t,
\end{align}
where $b :\R^d\to\R^d$ is measurable,
and $(W_t)_{t\ge0}$ is a $d$-dimensional Brownian motion supported on a filtered probability space
$(\Omega,\mathscr F,(\mathscr F_t)_{t\ge0},\P)$.

In most of scenarios,   SDEs under considerations   are unsolvable explicitly so various numerical approximation schemes are
proposed to discretize \eqref{eq1}  and implement the corresponding simulations  by the aid of  computers.  Inspired by e.g. \cite{BDMS,DEEGM,HJK}, we put forward
the following continuous-time approximation scheme associated with the SDE \eqref{eq1}:
for $\dd>0$ and $t\ge 0,$
\begin{align}\label{eq2}
	\d X^{(\dd)}_t=  b^{ (\delta) } (X^{(\dd)}_{t_\dd}) \,\d t+ \d W_t,
\end{align}
where   $t_\dd:=\lfloor t/\dd\rfloor \dd$ with $\lfloor\cdot\rfloor$  being the floor function, and $ b^{(\delta)}$ is an approximation version of $b$, namely, for each fixed $x\in\R^d,$ $|b^{(\delta)}(x)-b(x)|\to 0$ as  $\delta\to 0$. The scheme \eqref{eq2} is comprehensive so numerous competitive numerical algorithms
in the literature can be formulated as one of the special cases. In particular, the algorithm \eqref{eq2} incorporates the following typical candidates:
\begin{enumerate}
\item[(i)]  The EM scheme: $b^{(\delta)}=b$;

\item[(ii)] The truncated EM scheme (\cite{Mao}):  $b^{(\delta)}(x)=b(\pi^{\delta}(x))$, where $\pi^{\delta}:\R^d\to\R^d$ is a contractive truncation mapping, which is determined by  the local Lipschitz constant of $b$;

\item[(iii)] The tamed EM scheme (\cite{HJK,NMZ}): $b^{(\delta )}(x)=\frac{b(x)}{1+\delta^\theta|x|^{\ell_0}}$ or $b^{(\delta )}(x)=\frac{b(x)}{(1+\delta^{2\theta} |x|^{2\ell_0})^{\frac{1}{2}}}$, $x\in\R^d,$ where $\theta\in(0,1/2)$ and   $b$ is of polynomial growth with the leading power $\ell_0+1$, that is,
    $|b(x)|\le C_0(1+|x|^{\ell_0+1})$ for some constants $C_0,\ell_0>0.$
\end{enumerate}
In turn,   the scheme \eqref{eq2} can be regarded as an abstract framework extracted from the schemes mentioned previously. In some occasions, we shall write $(X_t^{(\delta),\mu})_{t\ge0}$ instead of $  (X_t^{(\delta) })_{t\ge0}$ once the initial distribution $\mu=\mathscr L_{X^{(\delta)}_0}$ is to be emphasized.

In order to widen the scope of the theory to be derived, we decompose  $b$ into two parts, written as  $b_0$   and $b_1$,
namely, $b=b_0+b_1$, where $b_1:\R^d\to\R^d$  and $b_0 :\R^d\to\R^d$ is allowed to possess  low regularity.
Concerning the drifts  $ b_0$ and $b_1$, we assume that
\begin{enumerate}
\item[$({\bf H}_1)$] $\R^d\ni x\mapsto b_1(x)$ is   locally Lipschitz  and
 there exist constants $K_1,K_2>0$ and $\alpha\in(0,1)$ such that for all $x,y\in \R^d,$
\begin{align*}
	\<x-y,b_1(x)-b_1(x)\>\le  K_1|x-y|^2\quad \mbox{ and } \quad  |b_0(x)-b_0(y)|\le K_2 |x-y|^\alpha;
\end{align*}

\item[$({\bf H}_2)$] there is  a   compact  function $V :\R^d\to [1,\8)$  and  constants $\ll_V,C_V>0$ such that for all $x\in  \R^d,$
\begin{align}\label{Lya1}
(\mathscr LV)(x)\le -\ll_V V(x)+C_V,
 \end{align}
 where $\mathscr L$ denotes the infinitesimal generator of $(X_t)_{t\ge0}$ governed by  \eqref{eq1};
 \item[$({\bf H}_3)$] for  $V$ given  in $({\bf H}_2)$, there exist  constants $L_V>0,\eta\in[0,1)$  such that for all $x,y\in  \R^d$ and $\delta\in(0,1],$
 \begin{align}\label{WE}
 \<\nn V(x), b^{(\delta)}(x)\>\le L_V(1+V(x)),
 \end{align}
 \begin{equation}\label{EW-}
 \|\nn ^2 V(x)\|_{\rm HS}\le L_V(1+V(x)^\eta),
 \end{equation}
and
\begin{align}\label{JJ}
(V(x)^\eta  |y|^2)\vee\sup_{u\in[0,1]}V(y+u(x-y))\le L_V(1+V(x)+V(y));
\end{align}
 \item[$({\bf H}_4)$]there exist constants $c^*,c_*>0$ and $\theta\in(0,1/2]$ such that for all   $x\in  \R^d $ and  $\delta\in(0,1],$
 \begin{align*}
 	 |b^{(\delta)}(x)|\le c_*+c^*\delta^{-\theta}|x|.
 \end{align*}
\end{enumerate}

 Under $({\bf H}_1)$, \eqref{eq1} admits a unique strong solution $(X_t)_{t\ge0}$ by appealing to a cut-off procedure
 along with an application of  the Zvonkin transformation; see, for example, \cite[Proposition 2.1]{BH} for related details. Additionally, from time to time, 
 we   write $(X_t^\mu)_{t\ge0}$ in lieu of $(X_t)_{t\ge0}$ in order to highlight the initial distribution $\mathscr L_{X_0}=\mu.$

 Below, we make some comments on assumptions $({\bf H}_1)$-$({\bf H}_4)$.
 \begin{remark}\label{Rem*}
 \begin{enumerate}
 \item[(i)]
 $({\bf H}_1)$ reveals that $b_1$ satisfies the so-called one-sided Lipschitz condition, which plays a crucial role when the coupling approach
 is adopted to handle the long-time behavior of $(X^{(\dd)}_t)_{t\ge0}$ as shown in the proof of Theorem \ref{thm1}, and that $b_0$ is allowed to enjoy  the H\"older regularity.  $({\bf H}_2)$ demonstrates that $(X_t)_{t\ge0}$ fulfills the Lyapunov condition (so, in the present work the SDE \eqref{eq1} is called a non-dissipative SDE in contrast to the existing literature \cite{BMW,BH3,EM,GWXY,HMW,LMP,IS,NMZ,SW}), which is indispensable
 to establish the uniform-in-time estimates between the exact solution and the associated numerical approximation scheme under the multiplicative type Wasserstein distance.

 \item[(ii)] As far as  the modified EM scheme \eqref{eq2}, which is much more general,  is concerned,  the first issue we need to address is the corresponding well-posedness.  For this, (${\bf H}_3$) and (${\bf H}_4$) are in force; see  Lemma \ref{lem0}  for more details.
Provided that the following assumption: for some constants $c_1,c_2>0$ and all $\delta\in(0,1],$
\begin{align}\label{RTT}
\<x,b^{(\delta)}(x)\>\le c_1+c_2|x|^2,\quad x\in\R^d
\end{align}
holds true, \eqref{WE} is valid for $V(x)=V_p(x)=  (1+|x|^2)^{\frac{1}{2}p}  $ with $p>0.$ The hypothesis \eqref{RTT} can be verified for the scheme \eqref{MTEM} and the scheme \eqref{TEM}, respectively; see the proof of Theorem \ref{pro1} for more details.  Apparently, \eqref{EW-} is verifiable for $V(x)=V_p(x)$ defined above. \eqref{JJ} shows that $V$ is at least of quadratic growth, which seems to be a bit stringent. Nevertheless, for some explicit  EM type schemes (e.g., the scheme \eqref{MTEM} and  the scheme \eqref{TEM}), the condition \eqref{JJ} can be abandoned as shown in Lemma \ref{lem0-1}. Furthermore, the precondition \eqref{EW-} is necessary to dominate the underlying quadratic process; see in particular the estimate of $\Gamma_2^\vv$ in Lemma \ref{lem3}.    It should be pointed out  that \eqref{EW-} holds true automatically in case $V(x)=V_p(x)$ with $p\in(2,\8 )$, or $\nn V$ (for instance, $V(x)=V_p(x)$ with $p\in(0,2]$) is globally Lipschitz continuous; see the estimate \eqref{JJ-} for related details.

  \item[(iii)]
 $({\bf H}_4)$ demonstrates that the modified drift $b^{(\delta)}$ is of linear growth, where the slope is dependent on the step size (i.e., $\delta^{-\theta}$,  which obviously tends to infinity as $\delta\to0$). In addition,  $({\bf H}_4)$ has been examined in the proof of Theorem \ref{pro1} for the scheme \eqref{MTEM} and the scheme \eqref{TEM} in Subsection \ref{sub1}, separately.
 \end{enumerate}
 \end{remark}

To measure the distance between two probability measures, we recall  from e.g.  \cite[(4.3)]{HMS}
the notion on the quasi-Wasserstein distance.  Let $\mathcal H$ be a Polish space and $\rho:\mathcal H\times\mathcal H\to[0,\infty)$
be a distance-like function (i.e., it is symmetric, lower semi-continuous, and such that $\rho(x,y)=0\Leftrightarrow x=y$; see, for instance,  \cite[Definition 4.3]{HMS}). The quasi-Wasserstein distance $\mathcal W_\rho$, induced by distance-like function $\rho$, is defined as below:
\begin{align*}
\mathcal W_\rho(\mu,\nu)=\inf_{\pi\in\mathscr C(\mu,\nu)}\int_{\mathcal H\times \mathcal H}\rho(x,y)\pi(\d x,\d y),\quad \mu,\nu\in\mathscr P(\mathcal H),
\end{align*}
where $\mathscr C(\mu,\nu)$ stands for the collection of couplings of $\mu,\nu$ and $\mathscr P(\mathcal H)$
embodies the family of probability measures on $\mathcal H.$ For a $\mu$-integrable function $f:\R^d\to\R$, we 
adopt  the shorthand $\mu(f)=\int_{\R^d}f(x)\mu(\d x)$.

The main result in this subsection is stated as follows.
\begin{theorem}\label{thm1}
	Assume $({\bf H}_1)$-$({\bf H}_4)$. Then, there exist constants
	$\ll_0,C_0,C_0^\star>0$ such that for all $\mu,\nu \in  \{\mu' \in \mathscr{P}(\mathbb{R}^d): \mu'(V)<\infty\}$, $\dd\in (0,1)$, and $t\ge0,$
	\begin{equation}\label{Q1}
		\begin{aligned}
	\mathcal W_{\rho_V}\Big(\mathscr L_{X^{(\dd),\mu}_t},\mathscr L_{X^{\nu}_t}\Big)\le &C_0\e^{-\ll_0 t}	\mathcal W_{\rho_{V}}(\mu,\nu)+C_0^\star\int_{0}^t\e^{-\ll_0(t-s)}\limsup_{\vv\to0}\E R^{\dd,\vv}_s\,\d s,
	\end{aligned}
	\end{equation}
in which
\begin{equation}\label{R1}
\rho_{V}(x,y):=(1\wedge|x-y|)(1+ V(x)+V(y)),\quad x,y\in\R^d,
\end{equation}
and
\begin{align}\label{F3} 
R^{\dd,\vv}_t:=\big|b(X^{(\dd,\vv)}_t)-b^{(\dd)} (X^{(\dd,\vv)}_{t_\dd})\big| \big(1+  V(X^{ \dd,\vv}_t)+ V(X^{(\delta,\vv)}_t)\big)
\end{align}
with $(X_t^{ \delta,\vv},X^{(\delta,\vv)}_t)_{t\ge0}$ being the coupling process determined in \eqref{eq3}. 
 \end{theorem}

In order to proceed with our analysis, we present some comments on Theorem \ref{thm1}.	
 	
 \begin{remark}\label{remark-}
In fact, constants $\ll_0,C_0,C_0^\star>0$ involved  in Theorem \ref{thm1} are given explicitly as follows:
 \begin{align*}
		C_0=1+\gg^{-1}l_0(1+c_2/c_1), \quad C_0^\star=\ff{2(2c_2\vee c^\star )}{c_1\gamma }, \quad\ll_0=\frac{c_3}{4(c_1+c_2)}\wedge\ff{\gg \ll_V}{  1+2\gg } ,
		\end{align*}
		where $$c^\star:=\gamma(1+c_1l_0)\big(\big(L_V(1+L_V^\eta)(1+2L_V)(1+V({\bf0})^\eta)\big)\vee |\nn V({\bf0})|\big),$$
		and
		$\gamma,c_1,c_2,c_3,l_0$ can also be written  precisely; see    \eqref{Q10} below for concrete expressions. Furthermore, at the first sight, \eqref{Q1} is not elegant   since there is a coupling process $(X_t^{ \delta,\vv},X^{(\delta,\vv)}_t)_{t\ge0}$, which is deferred to \eqref{EQ3}. 
 Nevertheless,  once   $b^{(\delta)}$ is given explicitly, the upper bound estimate of  $\E   R^{\dd,\vv}_t$  
 (which indeed is independent of the parameter $\vv$) boils down to the uniform-in-time  moment estimate of
$(X_t^{(\delta),\mu})_{t\ge0}$ and $(X_t^{\nu})_{t\ge0}$ since they are distributed identically as 
$(X_t^{(\delta,\vv)})_{t\ge0}$ and 
$(X_t^{ \delta,\vv })_{t\ge0}$, respectively; see  \eqref{F1} and \eqref{F2} below for some specifics. In addition, 
by H\"older's inequality, $\E R^{\dd,\vv}_t$ can be dominated obviously as below: 
\begin{align*}
\E R^{\dd,\vv}_t&\le \Big(\E\big|b(X^{(\dd,\vv)}_t)-b^{(\dd)} (X^{(\dd,\vv)}_{t_\dd})\big|^2\Big)^{\frac{1}{2}} \Big(1+  \big(\E V(X^{ \dd,\vv}_t)^2\big)^{\frac{1}{2}}+ \big(\E V(X^{(\delta,\vv)}_t)^2\big)^{\frac{1}{2}}\Big)\\
&=\Big(\E\big|b(X^{(\dd),\mu}_t)-b^{(\dd)} (X^{(\dd),\mu}_{t_\dd})\big|^2\Big)^{\frac{1}{2}} \Big(1+  \big(\E V(X^{ \nu}_t)^2\big)^{\frac{1}{2}}+ \big(\E V(X^{(\delta),\mu}_t)^2\big)^{\frac{1}{2}}\Big),
\end{align*}
where the identity holds true since $(X_t^{ \delta,\vv},X^{(\delta,\vv)}_t)_{t\ge0}$ is  a coupling process
$(X_t^\nu)_{t\ge0}$ and $(X_t^{(\delta),\mu})_{t\ge0}$. With the previous estimate at hand, the remainder involved in 
\eqref{Q1} can be reformulated via $(X_t^{(\delta),\mu},X_t^\nu)_{t\ge0}$ rather than the coupling process $(X_t^{ \delta,\vv},X^{(\delta,\vv)}_t)_{t\ge0}$. If we do it this way, the initial distributions $\mu,\nu$ need to satisfy $\nu(V^2)+\mu(V^2)<\8$ instead of $\nu(V)+\mu(V)<\8$ postulated in Theorem \ref{thm1}. Based on this, we prefer to write the remainder term at the price of introducing the coupling process $(X_t^{ \delta,\vv},X^{(\delta,\vv)}_t)_{t\ge0}$ instead of enhancing the integrability  of $V$ with respect to the initial distribution. 
 \end{remark}

 \begin{remark}\label{REm-}
 For the establishment of Theorem \ref{thm1}, in  \eqref{eq1}  we set the diffusion coefficient $\sigma=I_d$ (i.e., the identity  $d\times d$-matrix) for the sake of mere simplicity. Indeed, Theorem \ref{thm1} can be extended to the SDE \eqref{eq1} with   $\d W_t$ being  replaced by $\sigma\d W_t$, namely, 
\begin{align}\label{eq1-}
\d X_t=b(X_t)\,\d t+\si\d W_t,
\end{align}
 where $\sigma\in\R^d\otimes\R^d$ satisfies  $\si\si^\top\ge \lambda I_d$ for some constant $\lambda>0.$
 In this case, the corresponding $(X_t)_{t\ge0}$ is  distributed identically as $(Y_t)_{t\ge0}$, which is governed  by the following SDE:
 \begin{align}\label{EK}
 \d Y_t=b(Y_t)\,\d t+ \lambda^{\frac{1}{2}} \d W_t+(\si\si^\top-\lambda I_d)^{\frac{1}{2}}\d B_t,
 \end{align}
 where $(W_t)_{t\ge0}$ and $(B_t)_{t\ge0}$ are $d$-dimensional Brownian motions, which are mutually independent. With the SDE \eqref{EK} at hand,
 inspired by the coupling construction stated  in \eqref{eq3},  we roughly adopt the reflection coupling and the synchronous coupling
 when  the intensity  is a diagonal matrix and  a non-diagonal constant matrix, respectively.  Subsequently, for the new setup mentioned above, the corresponding proof of the restatement of Theorem \ref{thm1} can be finished by repeating the procedure adopted  in the present work.
 \end{remark}

\begin{remark}
One might wonder that  why we work merely  on the SDE \eqref{eq1} with the additive noise rather than the multiplicative noise. The key point lies in that, for such setting, the underlying SDE and its approximation version enjoy different diffusion terms so   the coupling approach adopted in the present work is no longer workable.
As regards the aforementioned framework, we will appeal to a completely different approach (namely, a combination of
an extended weak Harris' theorem and a Wasserstein coupling) to handle the corresponding non-asymptotic convergence analysis of the modified EM scheme.  This has been  investigated in depth in our forthcoming  paper   \cite{BHS}.
\end{remark}

\subsection{Application to  modified    EM schemes}\label{sub1}
In this subsection, we apply Theorem \ref{thm1} to
two   specific algorithms, namely,   the modified truncated EM scheme and the modified tamed EM algorithm.

In the previous subsection, some conditions are enforced on the modified drift $b^{(\delta)}$ (see \eqref{WE} and $({\bf H}_4)$).
In this subsection, we turn to impose  explicit hypotheses on the original drift $b$ so that \eqref{WE} and $({\bf H}_4)$
can be examined in a clear way.

More precisely,
regarding the drift term $b$, we suppose that
\begin{enumerate}
	\item[$({\bf A}_1)$] there exist constants $L_0>0, \ell_0\ge 0$ such that for all $x,y\in \R^d$,
	\begin{align*}
		|b_1(x)-b_1(y)|\le L_{0}(1+|x|^{\ell_0}+|y|^{\ell_0})|x-y|;
	\end{align*}
	\item [(${\bf A}_2$)] there exist constants $\ll^*,C^*>0$ such that for all $x\in \R^d,$
\begin{align*}
		\<x,b_1(x)\>\le -\ll^*|x|^{2}+C^*.
	\end{align*}
\end{enumerate}

It is standard  that the SDE \eqref{eq1} is strongly well-posed under $({\bf A}_1)$ and $({\bf A}_2)$; see e.g. \cite[Theorem 3.1.1.]{PR}.
In \eqref{eq2}, once we take $b^{(\delta)}(x)=b(\pi^{(\dd,\theta)}(x)) $ with
$$
 	\pi^{(\dd,\theta)}(x):=\big(|x|\wedge(\dd^{-\theta}-1)^{\ff{1}{\ell_0}}\big) \ff{x}{|x|}\mathds 1 _{\{|x|\neq 0\}}, \quad \theta\in(0,1/2),$$
the discrete-time version of \eqref{eq2} can be written    in the form: for $\delta>0$ and   $k\ge0,$
\begin{align}\label{H1-}
X^{(\delta)}_{(k+1)\delta}=X^{(\delta)}_{ k \delta}+ b(\pi^{(\dd,\theta)}(X^{(\delta)}_{k\dd})) \delta+ \triangle W_{k\delta},
\end{align}
where $\triangle W_{k\delta}:=W_{(k+1)\delta}-W_{k\delta}$.  As is known to all,  the  moment  boundedness in the infinite horizon  plays a crucial role
 in treating the long-time behavior of the scheme \eqref{H1-}.
 To obtain the uniform-in-time moment estimate
 of $(X^{(\delta)}_{k \delta})_{k\ge0}$, (${\bf A}_2$) is   imposed in general. Accordingly, one has
\begin{align}\label{H2}
		\<\pi^{(\dd,\theta)}(x),b_1(\pi^{(\dd,\theta)}(x))\>\le -\ll^*|\pi^{(\dd,\theta)}(x)|^{2}+C^*,\quad x\in\R^d.
	\end{align}
Whereas,   as far as the one-step iteration  in   \eqref{H1-} is concerned, only the drift term is modified nevertheless
the initial value is untouched. As a consequence, \eqref{H2} is insufficient to derive the uniform moment estimate by further noticing
that \eqref{H2} is not  true provided that  $|\pi^{(\dd,\theta)}(x)|^{2}$ (on the right hand side of  \eqref{H2}) is replaced by $|x|^2$. Apparently, the drift term $b_1$ can be reformulated   as follows:  for some  $a>0$,
\begin{align}\label{-H4}
b_1(x)=-ax+\bar b_1(x)  \quad \mbox{ with } \quad \bar b_1(x):=b_1(x)-ax,\quad x\in\R^d.
\end{align}
If $\bar b_1(x)$ is modified  into $\bar  b_1(\pi^{(\dd,\theta)}(x)) $, then we obtain the following approximate scheme:
\begin{align}\label{H3}
 X^{(\delta)}_{(k+1)\delta}=X^{(\delta)}_{ k \delta}-aX^{(\delta)}_{k \delta}\delta+  \bar  b_1(\pi^{(\dd,\theta)}(X^{(\delta)}_{k\dd})) \delta+ \triangle W_{k\delta}.
 \end{align}
 In case  the upper bound of  $\<x,\bar  b_1(\pi^{(\dd,\theta)}(x))\>$ is   a lower-order term in contrast to $-a|x|^2$ when $a$ is chosen appropriately and
  $|x|$ is large enough,
 the moment estimate  of  \eqref{H3}  in the infinite horizon can be obtainable. Based on the previous point of view, we construct the following   modified truncated  EM   scheme (see e.g. \cite{Mao} for the prototype) associated with \eqref{eq1}:
for $\dd\in(0,1)$,
\begin{align}\label{MTEM}
\d X^{(\delta)}_t=\big(b_0(X^{(\delta)}_{t_\dd})-\ll^*X^{(\delta)}_{t_\dd}+\bar  b_1(\pi^{(\dd,\theta)}(X^{(\delta)}_{t_\dd}))	\big)\,\d t+\d W_t,
\end{align}
 in which $\lambda^*>0$ is given in $({\bf A}_2)$ and $\bar b_1$ is defined in \eqref{-H4} with $a=\lambda^*.$

For  the tamed type EM algorithm, we generally  choose for $\theta\in(0,1/2],$ $$b^{(\delta,\theta)}_1(x)=\frac{b(x)}{1+\delta^\theta|x|^{\ell_0}},\quad x\in\R^d.$$
Most importantly, the parameter $\theta$ determines the  convergence rate of the corresponding scheme. In detail, the bigger the parameter  $\theta$ is, the faster the convergence rate is. Via a direct calculation, $({\bf A}_1)$ shows that there exists a constant $C_0>0$ such that for all $x\in\R^d$ and $\delta\in(0,1],$
\begin{align*}
 \delta|b^{(\delta,\frac{1}{2})}_1(x)|^2\le L_0 |x|^2+C_0.
\end{align*}
 As shown in Lemma \ref{lem0-1} below, the following condition: for some constants $\lambda,C_0^*>0$,
 \begin{align}\label{-H6}
\<x,b^{(\delta,\frac{1}{2})}_1(x)\>+\delta|b^{(\delta,\frac{1}{2})}_1(x)|^2\le -\lambda|x|^2+C_0^*
\end{align}
is of extreme importance to tackle  the uniform moment estimate of the associated scheme \eqref{eq2} for  $b^{(\delta)}_1=b^{(\delta,\frac{1}{2})}_1$. To guarantee that \eqref{-H6} is valid, an additional condition $\lambda^*>L_0$ has to be imposed, which is quite unreasonable.
 Alternatively, inspired by \cite{NMZ}, we can  take  for $\theta\in(0,1/2),$
 $$b^{(\delta,\theta)}_1(x)=\frac{b_1(x)}{(1+\delta^{2\theta}|x|^{2\ell_0})^{\frac{1}{2}}},\quad x\in\R^d.$$
Seemingly, $b^{(\delta,\theta)}_1$ defined in two different ways mentioned previously possesses  the same order in the denominator. Nevertheless, the latter one can provide a faster convergence rate; see, for instance,  \cite{NMZ} for more details.
In addition, to examine \eqref{-H6} in a very simple way,
we can further modify $b_1$ on account of \eqref{-H4} to  get the following one: for $\theta\in[\frac{1}{4},\frac{1}{2}),$
\begin{align}\label{-H7}
b^{(\delta,\theta)}_1(x)=-\lambda^*x+\frac{\bar b_1(x)}{(1+\delta^{2\theta}|x|^{2\ell_0})^{\frac{1}{2}}},\quad x\in\R^d.
\end{align}
 On the basis of the aforementioned viewpoint, we design the following modified tamed EM  scheme:
\begin{align}\label{TEM}
\d X^{(\delta)}_t=\big(b_0(X^{(\delta)}_{t_\delta})+b^{(\delta,\theta)}_1(X^{(\delta)}_{t_\delta})\big)\,\d t+\d W_t,
\end{align}
where $b^{(\delta,\theta)}_1$ is defined in \eqref{-H7}. For one more choice of $b^{(\delta,\theta)}_1$, we would like to refer  to \cite{LSZ}.

 With regard to the modified EM variants in \eqref{MTEM} and \eqref{TEM}, the following theorem addresses the issue on  non-asymptotic convergence bounds,
  where, most importantly,  the associated convergence rate is  provided at the same time.
 \begin{theorem}\label{pro1}
	Assume that $({\bf H}_1)$, $({\bf A}_1)$ and $({\bf A}_2)$ are satisfied. Then, concerning the modified truncated EM scheme \eqref{MTEM} and the modified  tamed EM scheme \eqref{TEM}, for any   $p>0,$ there exist constants {{$\ll^*,C^*, p^*,m^*>0,\dd^*\in (0,1)$}}such that for all $\mu\in\mathscr	P_{p^*}(\R^d),\nu \in \mathscr P_{2p}(\R^d)$ and $\dd\in(0,\dd^*)$,
	\begin{align}\label{pro1*}
	\mathcal W_{\rho_{|\cdot|^p}}\Big(\mathscr L_{X^{(\dd),\mu}_t},\mathscr L_{X^{\nu}_t}\Big)\le C^*\e^{-\ll^* t}	\mathcal W_{\rho_{|\cdot|^p}}(\mu,\nu)+C^*d^{m^*}\dd^{\ff\aa2},
	\end{align}
	where,  as for the scheme \eqref{MTEM},
	 $p^*=2(\ell_0+1+\ell_0/{(2\theta)}\vee p)$ and $m^*=\ell_0+1+\ell_0/{(2\theta)}\vee p;$  regarding the scheme \eqref{TEM},
$p^*=(1+3\ell_0)\vee(2p)$ and $m^*=(1/2+3\ell_0/2)\vee p.$
\end{theorem}

As an immediate application of Theorem \ref{pro1}, we obtain  the following corollary right now.
\begin{corollary}
	Assume that $({\bf H}_1)$, $({\bf A}_1)$ and $({\bf A}_2)$ are satisfied. Then, concerning the modified truncated EM scheme \eqref{MTEM} and the modified  tamed EM scheme \eqref{TEM}, for any   $p\ge1,$ there exist constants
$\ll^*,C^*, p^*,m^*>0,\dd^*\in (0,1)$ such that for all $\mu\in\mathscr	P_{p^*}(\R^d)$ and $\dd\in(0,\dd^*)$,
	\begin{align}
	\mathbb W_1\Big(\mathscr L_{X^{(\dd),\mu}_t},\pi_\8\Big)\le C^*\e^{-\ll^* t}	\mathcal W_{\rho_{|\cdot|^p}}(\mu,\pi_\8)+C^*d^{m^*}\dd^{\ff\aa2},
	\end{align}
	where $\mathbb W_1$  stands for  the standard $L^1$-Wasserstein distance and $\pi_\8$ is the invariant probability measure of $(X_t)_{t\ge0}.$
\end{corollary}
\begin{proof}
Note that $\mathcal W_1$ can be bounded by 	$\mathcal W_{\rho_{|\cdot|^p}}$ for $p\ge1$ and $\pi_\8$ has finite any  $p$-th moment by taking advantage of
Lemma \ref{lem4} below. Thus, the proof can be done by applying Theorem \ref{pro1} and choosing $\nu=\pi_\8$ therein.
\end{proof}

\subsection{Non-asymptotic convergence bound:  degenerate case}
As mentioned in Remark \ref{REm-}, we can also establish Theorem \ref{thm1} for the SDE \eqref{eq1-}, where the non-degenerate condition that
 $\si\si^\top\ge \lambda I_d$ is satisfied
for some constant $\lambda>0.$ In this subsection, we proceed to treat the  degenerate case: $\si\si^\top={\bf 0}$, a zero-matrix.
In particular, we shall consider  the following  degenerate SDE on $\R^{2d}:=\R^d\times\R^d:$ for any $t>0,$
\begin{align}\label{EQ1}
	\begin{cases}
 \d X_t= (aX_t+bY_t )\,\d t  \\
 \d Y_t=U(X_t,Y_t)\,\d t+ \d W_t,
\end{cases}
\end{align}
where $ a\ge 0,b>0$, $U:\R^{2d}\to \R^d$,
and $(W_t)_{t\ge 0}$ is a $d$-dimensional Brownian motion.

Suppose that
\begin{enumerate}
\item[$( {\bf B}_1)$]
 there exists a constant $L>0$ such that for all $x,x',y,y'\in \R^d,$
\begin{align*}
	| U(x,y)- U(x',y')|\le L\big(|x-x'|+|y-y'|\big);
\end{align*}
\item[$( {\bf B}_2)$]   there exist a   compact    $C^{2,2}$-function $V :\R^{2d}\to [1,\8)$, and constants $\ll_V^*,C_V^*>0$
such that for all $x,y\in \R^d$,
\begin{align*}
	(\mathscr LV)(x,y):&=\<\nn_1V(x,y), ax+by\>+\<\nn_2V(x,y),U(x,y)\>+\frac{1}{2}\Delta_2V(x,y)\\
&\le -\ll^*_V V(x,y)+C_V^*,
\end{align*}
where  $\nn_1 $ (resp. $\nn_2 $) means the gradient operator with respect to the first (resp. second) variable, and $\Delta_2$
represents the Laplacian operator with respect to the second variable;

\item[$( {\bf B}_3)$]   for $V$ in $( {\bf B}_2)$, there exist constants $L^*_V,L_V^\star>0$
and  $\eta\in[0,1)$ such that for all $x,y,x', y'\in \R^d$,
\begin{align}\label{a2}
	\big|\nn_2 V(x,y)-\nn_2 V(x',y')\big|\le L^*_V \big(V(x,y)^\eta+V(x',y')^\eta \big)\big(|x-x'|+|y-y'|\big)
\end{align}
and
\begin{align}\label{U19}
	|\nn_1V(  x,  y)|+|\nn_2 V(  x ,  y) |\le L_V^\star V(  x ,  y) .
\end{align}
\end{enumerate}

Under  $({\bf B}_1)$,   it is standard that the SDE \eqref{EQ1} is strongly well-posed (see e.g. \cite[Theorem 3.1, p.51]{Mao08})
and
the  following  EM scheme associated with \eqref{EQ1}: for any $t>0$ and $\delta\in(0,1],$
\begin{align}\label{EQ2}
	\begin{cases}
 \d X^{(\dd)}_t=(aX^{(\dd)}_{t_\dd}+bY^{(\dd)}_{t_\dd})\,\d t  \\
 \d Y^{(\dd)}_t=U(X^{(\dd)}_{t_\dd},Y^{(\dd)}_{t_
 \dd})\,\d t+ \d W_t
\end{cases}
\end{align}
is non-explosive in the $L^2$-sense in any  finite horizon.

 The following theorem provides a non-asymptotic convergence bound between $(X_t,Y_t)_{t\ge0}$ and $(X_t^{(\delta)},Y_t^{(\delta)})_{t\ge0}$ under the quasi-Wasserstein distance $\mathcal W_{\rho_V}$ which is induced by the distance-like function:  for any $(x,y),(x',y')\in\R^{2d}$,
 \begin{align}\label{U21}
 \rho_V((x,y),(x',y'))=\big(1\wedge(|x-x'|+|	y-y'|)\big)(1+V(x,y)+V(x',y')).
 \end{align}

\begin{theorem}\label{thm2}
	Assume that $({\bf B}_1)$-$({\bf B}_3)$ and suppose further there exists  $\delta^*\in(0,1]$ such that for each $t\ge0  $, $\delta\in(0,\delta^*] $ and $\nu\in\mathscr{P}(\mathbb{R}^{2d})$ satisfying $\mu(V)<\8,$
 \begin{align}\label{F10}
 \sup_{s\in[0,t]}\E V(X_{t}^{(\dd),\mu },Y_t^{(\dd),\mu })<\8.
\end{align}
Then, there exist constants
	$\ll_0,C_0,C_0^* >0$ such that for  $\mu,\nu\in \{\mu' \in \mathscr{P}(\mathbb{R}^{2d}): \mu'(\rho_V({\bf x},{\bf 0}))<\infty\},$ and $t\ge0,$
\begin{equation}\label{W12}
\begin{aligned}
 \mathcal W_{\rho_V}\Big(\mathscr L_{(X_{t}^{(\dd),\mu},Y_t^{(\dd),\mu})},\mathscr L_{(X_{t}^\nu,Y_t^\nu)} \Big) &\le C_0\e^{-\ll_0 t}	\mathcal W_{\rho_V}(\mu,\nu) +C_0^*\int_{0}^t\e^{-\ll_0(t-s)}\limsup_{\vv\to0}\E R_s^{\delta,\vv}\,\d s,
\end{aligned}
\end{equation}
 where for $t\ge0$ and $\vv>0$, 
 \begin{align*}
 R_t^{\delta,\vv}:=\big(V(   X^{\dd,\vv}_t,   Y^{\dd,\vv}_t)+ V(   X^{(\dd,\vv)}_t,   Y^{(\dd,\vv)}_t)\big) \big(  |  X^{(\dd,\vv)}_t- X^{(\dd,\vv)}_{t_\dd} |+  | Y^{(\dd,\vv)}_t- Y^{(\dd,\vv)}_{t_\dd} |\big)
 \end{align*}
 with  
 $((X^{\dd,\vv}_t,   Y^{\dd,\vv}_t),(X^{(\dd,\vv)}_t,Y^{(\dd,\vv)}_t))_{t\ge0}$ being the coupling process solving \eqref{EQ3}.
\end{theorem}

Below, we state some explanations concerning Theorem \ref{thm2}.

\begin{remark}
\begin{enumerate}
\item[(i)] Once the Lyapunov function $V$ is given, the estimate of $\E R_t^{\delta,\vv}$ comes down to
the uniform-in-time moment estimate of $(X_t^\nu,Y_t^\nu)_{t\ge0}$ and $(X_t^{(\delta),\mu},Y_t^{(\delta),\mu})_{t\ge0}$; see Remark \ref{remark-} and 
\eqref{F8} below for further details. 
As stated in Remark \ref{REm-}, Theorem \ref{thm2} is still valid whenever $\d W_t$ in \eqref{EQ1} is replaced by $\sigma\d W_t$, where $\sigma\sigma^\top$ is positive definite.
To prove Theorem \ref{thm2}, we still adopt the coupling approach. Since the dissipation in the $x$-direction is missing,
 as far as  the degenerate SDEs \eqref{EQ1} and \eqref{EQ2} are concerned,
the associated construction of coupling is quite different  from the one given in \eqref{eq3}.  To tackle the difficulty arising from the deficiency of dissipativity    in the $x$-direction, we have no alternative but to change the natural  metric in $\R^{2d}$ into a
comparable one, which in particular includes the cross term between the $x$-direction and the $y$-direction.
So, in \eqref{EQ3} below, we write $h_\vv(|x-x'+\alpha(y-y')|)$ and $\Pi(x-x'+\alpha(y-y'))$
 for a well-chosen $\alpha>0$  instead of $h_\vv(|x-x'|+|y-y'|)$ and $\Pi( y-y' )$, respectively,
 for all $(x,y),(x',y')\in \R^{2d} $. As for Langevin dynamics,  the coupling approach was applied in \cite{EGZ19} to the transformed version in order to investigate the weak contraction under the quasi-Wasserstein distance, where the gradient of the underlying   potential is also Lipschitz continuous. In contrast to the one provided  in  \cite{EGZ19}, the coupling constructed in the present work is much more direct since there is no a transform involved; see  \eqref{EQ3} for more details.

\item[(ii)] In Theorem \ref{thm1}, we focus on  the case that the drift term under consideration is of super-linear growth. Nonetheless, one might wonder why we are confined in Theorem \ref{thm2} to the setting that the underlying drift   is Lipschitz continuous (so it is of linear growth). To demonstrate this, we take the Hamiltonian function $H(x,y)=U(x)+\frac{1}{2}|y|^2$, $x,y\in\R^d,$ where $U:\R^d\to\R$
 is a smooth potential. Subsequently, the following underdamped Langevin    SDE:
 \begin{align}\label{E2-1}
	\begin{cases}
\d X_t=Y_t\,\d t  \\
 \d Y_t=-\big(\nn U(X_t)+Y_t\big)\,\d t+ \d W_t
\end{cases}
\end{align}
 is  available. Below, we   in particular choose $U(x)=\frac{1}{4}|x|^4-\frac{1}{2}|x|^2, x\in\R^d$, which is a well-known double well potential,
and write $(X^{(\delta)}_t,Y^{(\delta)}_t)_{t\ge0}$  determined by
the continuous-time tamed EM scheme corresponding to   the  SDE \eqref{E2-1}. To verify the moment boundedness of $(X^{(\delta)}_t,Y^{(\delta)}_t)_{t\ge0}$ in an  infinite horizon, we need to perturb the Hamiltonian  function $H$   into the Lyapunov function $V(x,y):=c_0+U(x)+\frac{1}{2}|y|^2+\alpha\<x,y\>$, where $c_0,\alpha>0$  are chosen so that $V$ is non-negative.
 With the preceding  $V$ at hand, it necessitates to estimate the quantity $|\<y,\nn U^\delta(x)\>|$ for the sake of the non-explosion in  the moment sense in an infinite horizontal, where $\nn U^\delta$ is the tamed version of $\nn U$. Note that $|\nn U^\delta(x)|\le c_1+c_2\delta^{-\theta}|x|$ for some constants $c_1,c_2>0 $  and that the prefactor $\delta^{-\theta}$ goes to infinity as the step size $\delta$
  approaches zero.  As a consequence, the term $|\<y,\nn U^\delta(x)\>|$  is out of control, and cannot be offset by the other dissipative terms on account of the difference in  the direction of variables.    Based on the aforementioned analysis, in the present work, we concern merely  the case that the drift under investigation is of linear growth. Nonetheless, the study on the long-term behavior of approximation schemes associated with underdamped Langevin SDEs with super-linear drifts  is a very interesting topic, which deserves to be investigated in depth in our future work. 
\end{enumerate}
\end{remark}

\label{ssub2}\subsection{Application to a kinetic Langevin sampler}
In this subsection, we apply Theorem \ref{thm2} to the  unadjusted Langevin sampler associated with \eqref{E2-1}, which
reads   as below:
\begin{align}\label{E2-2}
	\begin{cases}
 \d X^{(\dd)}_t=Y^{(\dd)}_{t_\dd}\,\d t  \\
 \d Y^{(\dd)}_t=-\big(\nn U(X^{(\dd)}_{t_\dd})+ Y^{(\dd)}_{t_\dd}\big)\,\d t+ \d W_t.
\end{cases}
\end{align}

For the potential $U $,  we assume that
\begin{enumerate}
\item[$( {\bf B}_4)$]  $U:\R^d\to\R$ is a $C^1$-function satisfying $\lim_{|x|\to\infty}U(x)=\infty, $ and
there exist constants $\ll_1,\ll_2,C^\star,L_U>0$ such that for all $x,y\in \R^d$,
\begin{align}\label{U22}
\<x,\nn U(x)\>\ge \ll_1 |x|^2+\ll_2 U(x)-C^\star,
\end{align}
and
\begin{align}\label{U-}
	|\nn U(x)-\nn U(y)|\le L_{U}|x-y|.
\end{align}
\end{enumerate}

The following theorem,  as an application of Theorem \ref{thm2},  reveals a non-asymptotic convergence bound between $( X_t,Y_t )_{t\ge0}$ solving \eqref{E2-1} and $(X_t^{(\delta)},Y_t^{(\delta)})_{t\ge0}$ determined by \eqref{E2-2}, and meanwhile furnishes the associated convergence rate.

\begin{theorem}\label{prolem}
Under $({\bf B}_4)$, there exist constants
	$\ll_0,C_0>0$ and $\delta_*\in(0,1]$ such that for all $\mu,\nu \in\mathscr P_3(\R^{2d})$, $\ll \in (0,2\ss{\lambda_1}\wedge\frac{1}{4})$, $\delta\in(0,\delta_*]$ and $t\ge0,$
\begin{equation}\label{U29}
	\mathcal W_{\rho_{  V_\lambda}}\Big(\mathscr L_{(X^{(\dd),\mu}_{t},Y^{(\dd),\mu}_{t})},\mathscr L_{(X^{\nu}_{t},Y_t^\nu)}\Big)\le C_0\big(\e^{-\ll_0 t}	\mathcal W_{\rho_{  V_\lambda}}(\mu,\nu)+\dd^{\ff12} d^{\ff32}\big),
\end{equation}
and
\begin{align}\label{U33}
\mathbb W_1\Big( \mathscr L_{(X^{(\dd),\mu}_{t},Y^{(\dd),\mu}_{t})}, \pi_\8\Big)\le 	C_0\big(\e^{-\ll_0 t}	\mathcal W_{\rho_{  V_\lambda}}(\nu,\pi_\8)+\dd^{\ff12} d^{\ff32}\big),
\end{align}
where $\pi_\8$ is the  invariant probability measure of $(X_t,Y_t)_{t\ge0},$ and $\rho_{V_\lambda}$ is defined  as in  \eqref{U21} with $V$ therein being replaced by $  V_\lambda$, defined as below
\begin{align}\label{F7}
 V_\lambda(x,y):=1+U(x)-\min_{x\in\R^d}U(x)+\ff14\big(|x+y|^2+|y|^2-\ll|x|^2\big),\quad (x,y)\in\R^{2d}.   	
\end{align}
\end{theorem}

In the end of this subsection, we make some further comments.
 \begin{remark}	 	
Under \eqref{U-}, along with the partially dissipative condition imposed on $\nn U,$ the $L^1$-Wasserstein contraction of three kinetic Langevin samplers and the quantitative error bound   between SGD with momentum and underdamped Langevin diffusion were studied  in \cite{SW} and \cite{GWXY}, respectively. Whereas, Theorem \ref{prolem} is established under the Lyapunov type condition (see \eqref{U22} for detail) instead of the partially dissipative condition enforced in \cite{GWXY,SW}. Whereas, the $1/2$-order convergence rate of the non-asymptotic $L^1$-Wasserstein error estimate  can still be achieved.
 \end{remark}

The rest of this paper is unfolded as follows. In Section \ref{sec2}, we first show that the modified EM scheme 
\eqref{eq2} is well-posed (see Lemma \ref{lem0}).  After that, via constructing an asymptotic coupling by reflection, 
the proof of Theorem \ref{thm1} is finished  based on several preliminary lemmas (see Lemma  \ref{lem1}, Lemma \ref{Lemma1}, Lemma \ref{lem2} and Lemma \ref{lem3}) concerned  with the radial process on the coupling process. Next,
 we provide a criterion to examine that the scheme \eqref{eq2} has a uniform-in-time moment estimate (see Lemma \ref{lem0-1}). Subsequently, concerning the modified truncated/tamed EM scheme, 
 the proof of Theorem \ref{pro1} is completed by checking the prerequisites enforced in Theorem \ref{thm1} and 
 Lemma \ref{lem0-1}, one by one. In Section \ref{sec3}, by the aid of a novel mixture of the synchronous coupling and the reflection coupling and taking advantage of  estimates (see Lemma \ref{lem7}, Lemma \ref{lem8} and Lemma \ref{lem9}) related to the transformed radial process on the coupling process,  we implement  the proof of Theorem \ref{thm2}. In addition, Section \ref{sec3} is also devoted to the proof of Theorem \ref{prolem} by invoking the  
 uniform moment estimate (see Lemma \ref{llem}) of the scheme \eqref{E2-2}.

\section{Proofs of Theorem \ref{thm1} and Theorem \ref{pro1}}\label{sec2}
The proof of Theorem \ref{thm1} is based on the coupling approach (see e.g.  \cite{DEGZ,Eberle}). Motivated by the asymptotic couplings by reflection designed  in \cite{Wang15} to handle the gradient/H\"older estimates as well as the exponential convergence of the Markov semigroups associated with monotone SPDEs, we aim to construct a mixed reflection coupling for two SDEs with different drifts. To this end, we prepare for some warm-up materials.  Let $(\bar W_t)_{t\ge0}$ and $(\tilde W_t)_{t\ge0}$ be independent copies of $(W_t)_{t\ge0}$. For a given parameter $\vv\in(0,1]$, let $h_{\varepsilon}:[0, \infty) \rightarrow[0,1]$ be  a continuous function such that $h_\vv(r)\equiv0$ for $ r\in[0,\vv],$ and $h_\vv(r)\equiv1$ for  $r\in [2\vv,\8),$ and set $h^*_\vv(r):=\ss{1-h_\vv(r)^2}$, $r\ge0.$
As an illustrative example, we  choose
\begin{align}\label{H1}
h_{\varepsilon}(r)\equiv0, ~  r\in[0,\vv]; \quad  h_{\varepsilon}(r)= 1-\mathrm{e}^{-\frac{r-\varepsilon}{2 \varepsilon-r}},~   r \in(\varepsilon, 2 \varepsilon); \quad h_{\varepsilon}(r)\equiv1,  ~ r \in [ 2 \varepsilon,\8).
\end{align}
Let $\R^d\ni x\mapsto 	\Pi(x)$ be an orthogonal $d\times d$-matrix defined by
\begin{align}\label{W3}
	\Pi(x)=I_{d }-2{\bf n}(x){\bf n}(x)^\top \mathds 1_{\{|x|\neq 0\}},
\end{align}
 where  the unit vector ${\bf n}(x):=x/|x|$ for a non-zero vector $x\in\R^d$,
 and $x^\top$ denotes the transpose of $x\in\R^d.$ With the previous notation,
we consider the following coupled SDE: for $\vv>0$ and $t\ge0,$
\begin{align}\label{eq3}
\begin{cases}
\d  X^{ \delta,\vv}_t= b(   X^{ \delta,\vv}_t)  \,\d t+ h_\vv(|Z^{ \dd,\vv }_t|)\, \d  \bar W_t+  h^*_\vv(|Z^{ \dd,\vv }_t|)\,\d \tilde W_t,\\
\d  X^{(\dd,\vv)}_t= b^{(\delta)}( X^{(\dd,\vv)}_{t_\dd}) \,\d t+ h_\vv(|Z^{ \dd,\vv }_t|)\Pi(Z^{ \dd,\vv }_t)\, \d   \bar W_t+ h^*_\vv(|Z^{(\dd,\vv)}_t|)\,\d \tilde W_t
\end{cases}
\end{align}
with the initial value $(X^{\dd,\vv }_0, X^{(\dd,\vv)}_0)=(X_0,X_0^{(\delta)})$,
where $Z^{ \dd,\vv }_t:= X^{\dd,\vv }_t- X^{(\dd,\vv)}_t$. Via L\'{e}vy's characterization of Brownian motions, besides  $h_\vv(r)^2+h_\vv^*(r)^2=1, r\ge0,$ it follows that
\begin{align*}
B_t:= \int_0^th_\vv(|Z^{ \dd,\vv }_s|)\, \d  \bar W_s+  \int_0^th^*_\vv(|Z^{ \dd,\vv }_s|)\,\d  \tilde W_s,\quad t\ge0
\end{align*}
is a standard Brownian motion. Furthermore, $\int_0^t\Pi(Z^{\dd,\vv}_s)\, \d   \bar W_s, t\ge0,$ is a Brownian motion, which is independent of $(\tilde W_t)_{t\ge0}$. Then,  under $({\bf H}_1)$,  the SDE \eqref{eq3} is weakly well-posed (see e.g. \cite[Theorem 1]{HS}) by recalling that the SDE \eqref{eq1} is weakly well-posed.

The interpretation of the construction given in \eqref{eq3} is presented as below. In case the distance between $  X^{\dd, \vv }_t$ and $ X^{(\dd,\vv)}_t$ is smaller than $\vv$ (i.e., $|Z^{ \dd,\vv }_t|\le\vv$),
we make use of the synchronous coupling; we employ the reflection coupling as soon as $|Z^{ \dd,\vv }_t|\ge\vv$;  in between, a mixture of the synchronous  coupling and the reflection coupling is exploited.

Now, we compare the coupling constructed in \eqref{eq3} with the approximate reflection coupling adopted in \cite{Suzuki}.
\begin{remark}
As shown in \cite[Lemma 2.1]{BH2},   for   given $\vv>0,$ $( X^{\dd,\vv }_t, X^{(\dd,\vv)}_t)_{t\ge0}$ is a genuine coupling of $(X_t,X^{(\delta)}_t)_{t\ge0}$. Nevertheless, for fixed $\vv>0$,  the coupled process $(X_t,Y^{\vv}_t)_{t\ge 0}$ given in \cite[(4.1) and (4.3)]{Suzuki} is not a coupling process.
To show that the associated weak limiting process of $(X_t,Y_t^{ \vv })_{\vv>0}$ is the desired coupling process, plentiful additional work (see Lemma 3.1, Lemma 3.2 and Lemma 3.3 therein) needs to be done. Therefore, the coupling constructed in the present paper is essentially different from the counterpart in \cite{Suzuki}.
\end{remark}

\subsection{Proof of Theorem \ref{thm1}}

In the following context,  we prepare some preliminary  lemmas  to avoid that the proof of Theorem \ref{thm1} is too lengthy.

\begin{lemma}\label{lem0}
Under  $({\bf H}_3)$ and $({\bf H}_4)$, there exists a constant $C_0>0$ such that
for any $t\ge0$ and  $\delta\in(0,1]$,
\begin{align}\label{W5}
\E\big(V(X^{(\delta)}_t)\big|\mathscr F_0\big)\le \big(V(X^{(\delta)}_0) +C_0t\big)\e^{C_0t}.
\end{align}
This further implies that $(X^{(\delta)}_t)_{t  \ge0 }$ is not  explosive  almost surely.
\end{lemma}
\begin{proof}
Below, for notational  brevity, we set
\begin{align*}
\Lambda^\delta_t: =\int_0^t \big \<  \nn V(X^{(\delta)}_s) -  \nn V(X^{(\delta)}_{s_\dd}) , b^{ (\delta) } (X^{(\delta)}_{s_\dd})\big \>  \,\d s,\quad t\ge0.
\end{align*}
By It\^o's formula and Young's inequality, besides the fact that $\DD V(x)\le \ss d\|\nn^2 V(x)\|_{\rm HS}, x\in\R^d$,
it follows from \eqref{WE}   that for some constants $C_1,C_2>0,$
\begin{equation}\label{W6}
\begin{aligned}
\E\big(V(X^{(\delta)}_t)\big|\mathscr F_0\big)&=V( X^{(\delta)}_0)+\int_0^t  \E\Big(\big\<  \nn V(X^{(\delta)}_s) , b^{ (\delta) } (X^{(\delta)}_{s_\dd}) \big\>+\frac{1}{2} \DD V(X^{(\delta)}_s)\Big|\mathscr F_0\Big)  \,\d s \\
&=V( X^{(\delta)}_0)+\int_0^t\E\big(\big\<  \nn V(X^{(\delta)}_{s_\dd}) , b^{ (\delta) } (X^{(\delta)}_{s_\dd}) \big\>\big|\mathscr F_0\big)\,\d s\\  &\quad+\E\big(\Lambda^\delta_t\big|\mathscr F_0\big)+\frac{1}{2} \int_0^t\E\big(\DD V(X^{(\delta)}_s)\big|\mathscr F_0\big)\,\d s\\
&\le V( X^{(\delta)}_0)+C_1\int_0^t\big( \E\big( V(X^{(\delta)}_{s_\dd})\big|\mathscr F_0\big) +\E\big( {V(X^{(\delta)}_s)} \big|\mathscr F_0\big)\big)\,\d s  \\
&\quad+\E\big(\Lambda^\delta_t\big|\mathscr F_0\big) +C_2t.
\end{aligned}
\end{equation}
Note from the approximation scheme \eqref{eq2} that $\Lambda^\delta_t$ can be rewritten as below:
\begin{align*}
\Lambda^\delta_t&=\int_0^t \int_0^1\big< \big(\nn^2 V(X^{(\delta)}_{s_\dd}+u(X^{(\delta)}_{s_\dd}-X^{(\delta)}_s ))\big)  \\
&\qquad\qquad\times\big(b^{ (\delta) } (X^{(\delta)}_{s_\dd})(s-s_\delta)+W_s-W_{{s_\delta}}\big), b^{ (\delta) } (X^{(\delta)}_{s_\dd}) \big>  \,\d u\,\d s.
\end{align*}
Thus, by virtue of (${\bf H}_3$), (${\bf H}_4$) as well as $\theta\in(0,1/2] $,
 we deduce  from  H\"older's  inequality and  that there exist constants $C_3,C_4>0$ such that for $\delta\in(0,1]$,
\begin{align*}
\E\big(\Lambda^\delta_t\big|\mathscr F_0\Big)&\le L_V\int_0^t\E\Big(\big(1+V(X^{(\delta)}_{s_\dd}+\aa(X^{(\delta)}_{s_\dd}-X^\dd_s ) )^\eta\big)\\
&\qquad\qquad\qquad\times\big(|b^{ (\delta) } (X^{(\delta)}_{s_\dd})|^2(s-s_\delta+\delta/2)+| W^\dd_s|^2/2\big)\big|\mathscr F_0\Big) \,\d s\\
&\le L_V(1+L_V^\eta) \int_0^t\E\Big(\big(1+V(X^{(\delta)}_{s_\dd})^\eta+V(X^{(\delta)}_s ) ^\eta\big)\big(  3(c^*)^2+3c_*^2\dd^{1-2\theta}|X^{(\delta)}_{s_\dd}|^2\\
&\qquad\qquad\quad\quad + | W^\dd_s|^2/2\big)\big|\mathscr F_0\Big)\,\d s\\
 &\le C_3\int_0^t  \E\Big(\big(\big(1+V(X^{(\delta)}_{s_\dd})^\eta+V(X^{(\delta)}_s  ) ^\eta\big)(1+|X^{(\delta)}_{s_\dd}|^2)\big)\big|\mathscr F_0\Big)\,\d s\\
 &\quad+C_3\int_0^t\E \big(V(X^{(\delta)}_{s_\dd})^\eta\big|\mathscr F_0\big)\,\d s+ C_3\int_0^t  \Big(\E\big(V(X^{(\delta)}_s )\big|\mathscr F_0\big)\Big)^{\eta}\Big(\E|W^\dd_s|^{\ff2{1-\eta}}\Big)^{1-\eta}\,\d s
 \\
&\le C_4\int_0^t \E\big(V(X^{(\delta)}_{s_\dd})+V(X^{(\delta)}_s)\big|\mathscr F_0\big)\,\d s+C_4t,
\end{align*}
where $W^\delta_s:=(W_s-W_{s_\delta})(s-s_\delta)^{-\frac{1}{2}}\sim N({\bf 0}, I_d)$ due to the scaling property of  $(W_t)_{t\ge0}$,  
 in the third inequality we used that $W_s-W_{s_\delta}$ is independent of $X^\dd_{s_\dd}$,and  the last inequality was valid  owing to the fact that for any integer $k\ge1,$
\begin{align*}
\E|W_{{s-s_\delta}}|^{2k}\le\frac{(2k)!((s-s_\delta)d)^k}{2^kk!},\quad s\ge0.
\end{align*}
Based on the previous estimates,
there exists a  constant  $ C_5>0$ such that
\begin{align*}
\E\big(V(X^{(\delta)}_t)\big|\mathscr F_0\big)
\le V(X^{(\delta)}_0)+C_5\int_0^t \E\big(V(X^{(\delta)}_{s_\dd})+V(X^{(\delta)}_s)\big|\mathscr F_0\big)\,\d s+C_5t.
\end{align*}
Subsequently, \eqref{W5} follows from Gronwall's inequality and  by noting that
\begin{align*}
\sup_{0\le s\le t}\E\big(V(X^{(\delta)}_s)\big|\mathscr F_0\big)
&\le V(X^{(\delta)}_0)+C_5t+2C_5\int_0^t\sup_{0\le u\le s} \E\big(V(X^{(\delta)}_s)\big|\mathscr F_0\big)\,\d s.
\end{align*}

Below, we fix $X_0^{(\delta)}=x\in\R^d.$ For any $n\ge1+|x|$, we define the stopping time:
\begin{align*}
\tau_n=\inf\big\{t\ge0: |X^{(\delta)}_t|\ge n\big\}.
\end{align*}
It is easy to see that $n\mapsto \tau_n$ is increasing so $\tau_\8:=\lim_{n\to\8}\tau_n$ is well defined. To show that $(X^{(\delta)}_t)_{t\ge0}$ does not explode almost surely, it is sufficient to verify $\P(\tau_\8=\8)=1$. To this end, we need   to prove that $\lim_{n\to\8}\P(\tau_n\le t)=0$
for arbitrary $t\ge0.$ Indeed, by means of  \eqref{W5}, we have  for all $t\ge0,$
\begin{align*}
\E V(X^{(\delta)}_{t\wedge\tau_n}) \le  (V(x)+c_0t )\e^{c_0t}.
\end{align*}
This obviously implies that for all $t\ge0,$
\begin{align*}
\inf_{|y|\ge n}V(y)\P(\tau_n\le t)\le  (V(x)+c_0t )\e^{c_0t}.
\end{align*}
Whence, for any $t\ge0$,   $\lim_{n\to\8}\P(\tau_n\le t)=0$  is available  by recalling that $V$ is a   compact  function.
\end{proof}

Next, we demonstrate that  the radial process $(| Z^{ \dd,\vv }_t|)_{t\ge0} $ satisfies an SDE, where it is extremely important that the corresponding quadratic process vanishes.

\begin{lemma}\label{lem1}
For any $\delta,\vv,t>0$,
\begin{equation}\label{Q2}
\d  |Z^{ \dd,\vv }_t|
 = \I_{\{|Z^{ \dd,\vv }_t|\neq0\}}  \<{\bf n}(Z^{ \dd,\vv }_t),b(  X^{\dd,\vv}_t)-b^{(\delta)}(X^{(\dd,\vv)}_{t_\dd}) \>   \,\d t +2 h_\vv(|Z^{ \dd,\vv }_t|) \I_{\{|Z^{ \dd,\vv }_t|\neq0\}}\<{\bf n}(Z^{ \dd,\vv }_t),\d \bar W_t\>.
\end{equation}
\end{lemma}

\begin{proof}
 By It\^o's formula, it follows that for any $t>0,$
\begin{align*}
\d  |Z^{ \dd,\vv }_t|^2&=\big(2\<Z^{ \dd,\vv }_t,b(  X^{\dd,\vv}_t)-b^{(\delta)}( X^{(\dd,\vv)}_{t_\dd}) \>+4h_\vv(|Z^{ \dd,\vv }_t |)^2\big)\,\d t+4h_\vv(|Z^{ \dd,\vv }_t|)\<Z^{ \dd,\vv }_t,\d \bar W_t\>.
\end{align*}	
In the sequel, for given $\kk>0,$ let $\phi_\kk(r)=(\kk+r)^{1/2},  r\ge0.$ It is easy to see that
$$\phi_\kk'(r)=\frac{1}{2}\phi_\kk(r)^{-1}\quad \mbox{ and } \quad \phi_\kk''(r)=-\frac{1}{4}\phi_\kk(r)^{-3},\quad r\ge0.$$
Subsequently, using $\phi_\kk(r)^{-1}-\phi_\kk(r)^{-3}r=\kk \phi_\kk(r)^{-3}, r\ge0$,
 and applying It\^o's formula yields that
\begin{equation*}
\begin{split}
\d \phi_\kk(|Z^{ \dd,\vv }_t|^2)&=\bigg(\frac{ \<Z^{ \dd,\vv }_t,b(  X^{\dd,\vv}_t)-b^{(\delta)}(  X^{(\dd,\vv)}_{t_\dd}) \> }{ \phi_\kk(|Z^{ \dd,\vv }_t|^2)}+\frac{ 2\kk h_\vv(|Z^{ \dd,\vv }_t|)^2}{\phi_\kk(|Z^{ \dd,\vv }_t|^2)^3}\bigg)\,\d t+\frac{2h_\vv(|Z^{ \dd,\vv }_t|)}{ \phi_\kk(|Z^{ \dd,\vv }_t|^2)}\<Z^{ \dd,\vv }_t,\d \bar W_t\>.
\end{split}
\end{equation*}
In terms of the definitions of $h_\vv$ and $\phi_\kk$, we have
\begin{align*}
\frac{\kk h_\vv(r)^2}{\phi_\kk(r^2)^3}=0 , \quad  r\le\vv,  \mbox{ and }  \quad \frac{\kk h_\vv(r)^2}{\phi_\kk(r^2)^3} \le \frac{\kk}{\vv^3},\quad   r\ge\vv.
\end{align*}
Thus, the dominated convergence theorem enables us to derive that
\begin{align}\label{Q4}
\lim_{\kk\to0}\int_0^t\frac{  \kk h_\vv(|Z^{ \dd,\vv }_s|)^2}{\phi_\kk(|Z^{ \dd,\vv }_s|^2)^3} \,\d s=0.
\end{align}
Additionally, thanks  to $\lim_{\kk\to0}h_\vv(|x|)x/\phi_\kk(|x|)=h_\vv(x){\bf n}(x)\I_{\{|x|\neq0\}}$, we obtain from \cite[Theorem 2.12, p.142]{RY}  that
\begin{align}\label{Q5}
\lim_{\kk\to0}\int_0^t\frac{ h_\vv(|Z^{ \dd,\vv }_s|)}{ \phi_\kk(|Z^{ \dd,\vv }_s|^2)}\<Z^{ \dd,\vv }_s,\d \bar W_s\>\overset{a.s.}=\int_0^t  h_\vv(|Z^{ \dd,\vv }_s|) \I_{\{|Z^{ \dd,\vv }_s|\neq0\}}\<{\bf n}(Z^{ \dd,\vv }_s),\d \bar W_s\>.
\end{align}	
At length, \eqref{Q2} follows by  combining \eqref{Q4} with \eqref{Q5} and
 leveraging    $x/\phi_\kk(|x|^2)\to {\bf n}(x)\I_{\{|x|\neq0\}}$ as $\kk\to0,$
\end{proof}

Before we proceed, we introduce some additional notation. For $\gamma>0 $ and
a
concave function   $g:[0,\8)\to[0,\8)$, define
\begin{align}\label{RR-}
\rho_{g,\gamma,V}(x,y) =g(|x-y|)(1+\gamma V(x)+\gamma V(y)),\quad  x,y\in\R^d,
\end{align}
where $V$ is the Lyapunov function stipulated in $({\bf H}_2).$ Additionally, we set
\begin{equation}\label{Q7}
\psi_\vv(r):=g'_{-}(r)  (K_1 r +K_2 r^\alpha  )  +2 g''(r)  h_\vv(r)^2, \quad r\ge0,
 \end{equation}
where $g'_{-}$ is the left-hand derivative  and $g''$ the second order derivative defined almost everywhere, and for $\gamma,t\ge0,$
\begin{equation}\label{RR}
 \begin{split}
R^{\delta,\vv,\gamma}_t:&=|b(  X^{(\dd,\vv)}_t)-b^{(\delta)}(  X^{(\dd,\vv)}_{t_\dd})|\\
&\quad\times\big((1+\gamma V( X^{\dd,\vv}_t)+\gamma V( X^{(\dd,\vv)}_t))   g'_{-}(|Z^{ \dd,\vv }_t|) +\gamma g(|Z^{ \dd,\vv }_t|) |\nn V( X^{(\dd,\vv)}_t)|\big).
\end{split}
 \end{equation}

\begin{lemma}\label{Lemma1}
Assume   that $({\bf H}_1)$ and $({\bf H}_2)$ hold. Then, for any $\delta,\vv,\gamma,t>0,$
\begin{equation}\label{Q6}
\begin{split}
\d \rho_{g,\gamma,V}( X^{\dd,\vv}_t, X^{(\dd,\vv)}_t)   &\le \d M^{\delta,\vv,\gamma}_t+\Big((1+\gamma V( X^{\dd,\vv}_t)+\gamma V( X^{(\dd,\vv)}_t)) \psi_\vv (|Z^{ \dd,\vv }_t|) \\
&\quad+\gamma g(|Z^{ \dd,\vv }_t|)\big(-\lambda_V V(X^{\dd,\vv}_t)-\lambda_V  | V(X^{(\dd,\vv)}_t) + 2C_V   \big)\Big)\,\d t\\
&\quad+2\gamma g'_{-}(|Z^{ \dd,\vv }_t|)  h_\vv(|Z^{ \dd,\vv }_t|)^2 \I_{\{|Z^{ \dd,\vv }_t|\neq0\}}\\
&\quad\quad\times\<{\bf n}(Z^{ \dd,\vv }_t), \nn V(X^{\dd,\vv}_t)+\Pi(Z^{ \dd,\vv }_t)\nn V(X^{(\dd,\vv)}_t)\>\,\d t+R^{\delta,\vv,\gamma}_t\,\d t,
\end{split}
\end{equation}
  where $\rho_{g,\gamma,V}$, $\psi_\vv$, and $R^{\delta,\vv,\gamma}$ are defined  in \eqref{RR-}, \eqref{Q7} and \eqref{RR}, respectively,  and
  $(M^{\delta,\vv,\gamma}_t)_{t\ge0}$ is a martingale.
\end{lemma}

\begin{proof}
Since  $g:[0,\8)\to[0,\8)$ is concave,   we obtain from Lemma \ref{lem1} and It\^o-Tanaka's formula (see e.g. \cite[Theorem 3.7.1]{KSh}) that
\begin{align*}
 g(|Z^{ \dd,\vv }_t|)&=g(|Z^{\dd,\vv}_0|)+\int_0^tg'_{-}(|Z^{ \dd,\vv }_s|)\I_{\{|Z^{ \dd,\vv }_s|\neq0\}}  \<{\bf n}(Z^{ \dd,\vv }_s),b( X^{\dd,\vv}_s)-b^{(\delta)}(X^{(\dd,\vv)}_{s_\dd}) \>   \,\d s\\
 &\quad+\frac{1}{2}\int_{-\8}^\8\Lambda^{\delta,\vv}(t,a)\mu_g(\d a) +2\int_0^tg'_{-}(|Z^{ \dd,\vv }_s|)  h_\vv(|Z^{ \dd,\vv }_s|) \I_{\{|Z^{ \dd,\vv }_s|\neq0\}}\<{\bf n}(Z^{ \dd,\vv }_s),\d \bar W_s\>,
\end{align*}
where $\mu_g$ means the second derivative measure (see e.g. \cite[(6.47)]{KSh}), and $\Lambda^{\delta,\vv}(t,a)$ is the local time of $(|Z^{ \dd,\vv }_t|)_{t\ge0}$ at the point $a.$ Thereafter, the occupation time formula  and the fact that $\mu_g(\d a)\le g''(a)\d a$ imply that
\begin{align*}
\int_{-\8}^\8\Lambda^{\delta,\vv}(t,a)\mu_g(\d a)\le 4\int_0^tg''(|Z^{ \dd,\vv }_s|)  h_\vv(|Z^{ \dd,\vv }_s|)^2 \I_{\{|Z^{ \dd,\vv }_s|\neq0\}}\,\d s.
\end{align*}
As a result,
 we derive from $g'_{-}\ge0 $ that
\begin{equation}\label{U7}
\begin{split}
 \d  g(|Z^{ \dd,\vv }_t|)&\le  g'_{-}(|Z^{ \dd,\vv }_t|)\I_{\{|Z^{ \dd,\vv }_t|\neq0\}}  \<{\bf n}(Z^{ \dd,\vv }_t),b( X^{\dd,\vv}_t)-b^{(\delta)}(X^{(\dd,\vv)}_{t_\dd}) \>   \,\d t\\
 &\quad+2 g''(|Z^{ \dd,\vv }_t|)  h_\vv(|Z^{ \dd,\vv }_t|)^2 \,\d t\\
 &\quad+2 g'_{-}(|Z^{ \dd,\vv }_t|)  h_\vv(|Z^{ \dd,\vv }_t|) \I_{\{|Z^{ \dd,\vv }_t|\neq0\}}\<{\bf n}(Z^{ \dd,\vv }_t),\d \bar W_t\>\\
 &\le\big(\psi_\vv (|Z^{ \dd,\vv }_t|) + g'_{-}(|Z^{ \dd,\vv }_t|)|b(  X^{(\delta,\vv)}_t)-b^{(\delta)}(X^{(\dd,\vv)}_{t_\dd})| \big)\,\d t\\
  &\quad+2 g'_{-}(|Z^{ \dd,\vv }_t|)  h_\vv(|Z^{ \dd,\vv }_t|) \I_{\{|Z^{ \dd,\vv }_t|\neq0\}}\<{\bf n}(Z^{ \dd,\vv }_t),\d \bar W_t\>,
\end{split}
\end{equation}
where $\psi_\vv$ is defined in \eqref{Q7}.
Next, applying It\^o's formula and making use of $({\bf H}_2)$  gives that
\begin{align*}
&\d \big(V(X^{\dd,\vv}_t)+V(X^{(\delta,\vv)}_t)\big)\\&=\big((\mathscr LV)(X^{\dd,\vv}_t)+(\mathscr LV)( X^{(\delta,\vv)}_t)\big)\,\d t +\< \nn V(X^{(\delta,\vv)}_t) ,  b^{(\delta)}(X^{(\dd,\vv)}_{t_\dd})-b(  X^{(\delta,\vv)}_t)\>\,\d t\\
&\quad+h_\vv(|Z^{ \dd,\vv }_t|)\<\nn V(X^{\dd,\vv}_t)+\Pi(Z^{ \dd,\vv }_t)\nn V(X^{(\dd,\vv)}_t), \d  \bar W_t \>\\
&\quad+h^*_\vv(|Z^{ \dd,\vv }_t|)\<\nn V(X^{\dd,\vv}_t)+\nn V( X^{(\dd,\vv)}_t),  \d  \tilde W_t\>\\
&\le  \big(-\lambda_V(V(X^{\dd,\vv}_t)+V( X^{(\delta,\vv)}_t))+ 2C_V\big)\,\d t +|\nn V( X^{(\delta,\vv)}_t )|\cdot |b^{(\delta)}(X^{(\dd,\vv)}_{t_\dd})-b( X^{(\delta,\vv)}_t)|\,\d t\\
&\quad+h_\vv(|Z^{ \dd,\vv }_t|)\<\nn V(X^{\dd,\vv}_t)+\Pi(Z^{ \dd,\vv }_t)\nn V(X^{(\dd,\vv)}_t), \d  \bar W_t \>\\
&\quad+h^*_\vv(|Z^{ \dd,\vv }_t|)\<\nn V( X^{\dd,\vv}_t)+\nn V(X^{(\dd,\vv)}_t),  \d  \tilde W_t\>.
\end{align*}
Finally, \eqref{Q6} is achievable by noting from the chain rule that
\begin{align*}
\d \rho_{g,\gamma,V}( X^{\dd,\vv}_t, X^{(\delta,\vv)}_t)&=(1+\gamma V(  X^{\dd,\vv}_t)+\gamma V( X^{(\delta,\vv)}_t))\d  g(|Z^{ \dd,\vv }_t|)\\
&\quad +\gamma g(|Z^{ \dd,\vv }_t|)\d \big(V(X^{\dd,\vv}_t)+V(X^{(\delta,\vv)}_t)\big)\\
&\quad+\gamma \d[g(|Z^{ \dd,\vv }_\cdot|), V(X^{\dd,\vv}_\cdot)+V(X^{(\delta,\vv)}_\cdot)]_t,
\end{align*}
where $[\cdot,\cdot]_t$ denotes the quadratic variation process, and
\begin{align*}
  \d[g(|Z^{ \dd,\vv }_\cdot|),V(X^{\dd,\vv}_\cdot)+V(X^{(\delta,\vv)}_\cdot)]_t &=2 g'_{-}(|Z^{ \dd,\vv }_t|)  h_\vv(|Z^{ \dd,\vv }_t|)^2 \I_{\{|Z^{ \dd,\vv }_t|\neq 0\}}\\
&\quad\times \<{\bf n}(Z^{ \dd,\vv }_t), \nn V(X^{\dd,\vv}_t)+\Pi(Z^{ \dd,\vv }_t)\nn V( X^{(\dd,\vv)}_t)\>
\end{align*}
by taking advantage of the independence between  $(\bar W_t)_{t\ge0}$ and $(\tilde W_t)_{t\ge0}$.
\end{proof}

 To derive the desired assertion \eqref{Q1}, besides the construction of the coupling given in \eqref{eq3}, another essential is to choose a suitable quasi-distance, which is comparable to $\rho_{V}$ defined in \eqref{R1}. To this end, we need to introduce some additional quantities. For the Lyapunov function $V$ and parameters $C_V,\lambda_V>0$
given in (${\bf H}_2$), we define the set
\begin{align*}
\mathcal D_V=\big\{(x,y)\in\R^d\times\R^d: V(x)+V(y)\le  4C_V/{\ll_V} \big\}.
\end{align*}
By recalling that $V$ is a compact  function,  $\mathcal D_V$ is a bounded set  so the quantity
$$l_0:=1+\sup_{(x,y)\in\mathcal D_V}|x-y|$$
is finite. In the following analysis,  we fix the following quantities:
\begin{equation}\label{Q10}
\begin{split}
c_1&=c_2\e^{-c_2 l_0},\quad c_2= 2(K_1l_0+K_2 l_0^{\alpha}),\quad c_3:=  c_2^2\e^{ -c_2l_0}/{l_0},\\
\gamma&=1\wedge \frac{c_3}{4C_V(c_1+c_2)}\wedge\bigg(\frac{ c_3}{24L_V(1+L_V^\eta)(c_1+c_2)(1+c_2/{c_1})}\bigg)^{\frac{1}{1-\eta}},
\end{split}
\end{equation}
 and choose
\begin{align}\label{Q8}
g(r) =f(r\wedge l_0),\quad  \mbox{ with }	f(r):=c_1 r+1-\e^{-c_2 r},\quad r\ge 0.
\end{align}

In the sequel, we seek to estimate the terms involved in the right hand side of \eqref{Q6}, separately.

\begin{lemma}\label{lem2}
For all $x,y\in\R^d,$
\begin{equation}\label{W1}
\begin{split}
\Gamma_1^{\vv}(x,y):&=(1+\gamma V(x)+\gamma V(y)) \psi_\vv (|x-y|)  +\gamma g(|x-y|)\big(-\lambda_V V(x)-\lambda_V  V( y) + 2C_V   \big)\\
&\le -\Big(\frac{c_3}{2(c_1+c_2)}\I_{\{|x-y|\le l_0\}} +\ff{\gg \ll_V}{  1+2\gg }\I_{\{|x-y|>l_0\}}\Big)\rho_{g,\gg,V}(x,y)\\
&\quad+2 \big( (c_3 + 2c_2  K_1)\vv +2K_2\vv^\alpha  \big)(1+\gamma V(x)+\gamma V (y)),
\end{split}
\end{equation}
where the parameters $c_1,c_2,c_3,\gamma >0$ are given  in \eqref{Q10}.
\end{lemma}

\begin{proof}
The   estimate \eqref{W1} is  derived on account of the cases (i) $x,y\in\R^d$ with $|x-y|\le \ell_0 $ and (ii) $x,y\in\R^d$ with $|x-y|> \ell_0.$

For the case (i),  $\Gamma_1^\vv $ can be reformulated obviously  as below:
\begin{align*}
\Gamma_1^\vv(x,y)&=(1+\gamma V(x)+\gamma V (y))\\
  &\quad\quad\times\big(h_\vv(|x-y|)^2\phi(|x-y|)+(1-h_\vv(|x-y|)^2)f'(|x-y|)(K_1|x-y|+ K_2 |x-y|^{\alpha})\big)\\
&\quad  +\gamma f(|x-y|)\big(-\lambda V(x)-\lambda  V( y) + 2C_V   \big),
\end{align*}	
where
\begin{align*}
\phi(r):=f'(r)(K_1r+ K_2 r^{\alpha})+2  f''(r),\quad r\ge0.
\end{align*}	
By virtue of
\begin{align*}
f'(r)=c_1+c_2\e^{-c_2r}\quad \mbox{ and } \quad f''(r)=-c_2^2\e^{-c_2r},\quad r\ge0,
\end{align*}
it follows that for any $r\in[0,l_0]$,
\begin{equation}\label{Q9}
		\begin{split}	 \phi(r)
		=(c_2\e^{-c_2 r}+c_1)(K_1r+ K_2 r^{\alpha})-2 c_2^2 \e^{-c_2 r}
			  \le 2c_2(K_1l_0+ K_2 l_0^{\alpha}-c_2  )\e^{-c_2 r} \le -c_3  r
	\end{split}
	\end{equation}
where in the first inequality we used $c_1=c_2\e^{-c_2l_0}$, and in the second inequality   we utilized $\frac{1}{2}c_2=K_1l_0+ K_2 l_0^{\alpha}$.
Thus, \eqref{Q9}, in addition to $f'(r)\le 2c_2, r\ge0,$ implies that for any $x,y\in\R^d$ with $|x- y|\le l_0,$
\begin{align*}
\Gamma_1^\vv(x,y)&\le-\big(c_3|x-y| -2\gamma  C_Vf(|x-y|)\big)(1+\gamma V(x)+\gamma V (y)) \\
&\quad+(1-h_\vv(|x-y|)^2)(1+\gamma V(x)+\gamma V (y))\big(c_3|x-y| +2c_2(K_1|x-y|+ K_2 |x-y|^{\alpha})\big)\\
&\le-\frac{c_3}{2(c_1+c_2)} \rho_{f,\gamma,V}(x,y)+\big(2c_3\vv+2c_2(2K_1\vv+K_2(2\vv)^\alpha)\big)(1+\gamma V(x)+\gamma V (y)),
\end{align*}	
where in the second inequality we exploited the fact that $f(r)\le (c_1+c_2)r,r\ge0,$  took $\gamma>0$ such that $\frac{c_3}{2(c_1+c_2)}=2\gamma C_V$, as well as made use of that
\begin{align*}
(1-h_\vv(r)^2) \big(c_3r +2c_2(K_1r+ K_2 r^{\alpha})\big)\le 2c_3\vv+2c_2(2K_1\vv+K_2(2\vv)^\alpha),\quad r\ge0.
\end{align*}

With regard to the case (ii), $g'_{-}(|x-y|)=g''(|x-y|)=0$ and $(x,y)\notin \mathcal D_V$
(so $V(x)+V(y)>4C_V/{\lambda_V}$)
so that
\begin{align*}
	\Gamma_1^{\vv}(x,y)\le-\frac{1}{2}\lambda_V\gamma g(|x-y|) (   V(x)+  V(y)   ) &\le-\ff{\gg\ll_V(V(x)+V(y))}{2\big(1+\gamma(V(x)+V(y))\big)}\rho_{g,\gg,V}(x,y)\\
		&\le -\ff{\gg \ll_V}{  1+2\gg }\rho_{g,\gg,V}(x,y),
	\end{align*}
in which  in the last display we employed the prerequisite $V\ge1$ and the fact that $[0,\8)\ni r\mapsto r/(1+r)$ is increasing.

At last, the assertion \eqref{W1} is available by summarizing the aforementioned estimates.
\end{proof}

\begin{lemma}\label{lem3}
For any $x,y\in\R^d$,
\begin{equation}\label{W2}
\begin{split}
\Gamma_2^{\vv}(x,y):&=2\gamma g'_{-}(|x-y|)  h_\vv(|x-y|)^2 \I_{\{|x-y|\neq0\}}\<{\bf n}(x-y), \nn V(x)+\Pi(x-y)  \nn V(y)\>\\
&\le\frac{c_3}{4(c_1+c_2)}\rho_{g,\gg,V}(x,y)\I_{\{|x-y|\le l_0\}}.
\end{split}
\end{equation}

\end{lemma}

\begin{proof}
Apparently,  for any $x,y\in\R^d$ with $|x-y|>l_0$, the assertion \eqref{W2} holds true by noting that $g'_{-}(|x-y|) =0 $.
So, it remains to  show that \eqref{W2} is still valid for any $x,y\in\R^d$ with $|x-y|\le l_0.$ Next, from  the definition of $\Pi$ given in \eqref{W3},
we have
\begin{align}\label{U9}
 \Pi(z){\bf n}(z)=-{\bf n}(z),\quad {\bf 0}\neq z\in\R^d.
\end{align}
This, along with $g'_{-}(|x-y|) \le c_1+c_2$ and $h_\vv\le1$, leads to
\begin{equation}\label{JJ-}
\begin{split}
\Gamma_2^{\vv}(x,y)&=2\gamma g'_{-}(|x-y|)  h_\vv(|x-y|)^2 \I_{\{|x-y|\neq0\}}\<{\bf n}(x-y), \nn V(x)- \nn V(y)\>\\
&\le 2\gamma(c_1+c_2)|\nn V(x)- \nn V(y)|\\
&\le 2L_V(1+L_V^\eta)\gamma(c_1+c_2)\big(1+V(x)^\eta+V(y)^\eta\big)|x-y|\\
&= 2L_V(1+L_V^\eta)\gamma(c_1+c_2)\big(1+(\eta^{\frac{1}{1-\eta}}\gamma^{-1})^{\eta(1-\eta)}(\gamma^{1-\eta}V(x)/\eta)^\eta\\
&\qquad\qquad\qquad\quad\quad+(\eta^{\frac{1}{1-\eta}}\gamma^{-1})^{\eta(1-\eta)}(\gamma^{1-\eta}V(y)/\eta)^\eta\big)|x-y|\\
&\le2L_V(1+L_V^\eta)\gamma(c_1+c_2)\\
&\quad\times\big(1+2(1-\eta)(\eta^{\frac{1}{1-\eta}}\gamma^{-1})^{\eta }+ \gamma^{1-\eta}V(x) + \gamma^{1-\eta}V(y) \big)|x-y|\\
&\le3L_V(1+L_V^\eta)\gamma^{1-\eta}(c_1+c_2)\big(2   + \gamma V(x) + \gamma V(y) \big)|x-y|,
\end{split}
\end{equation}
where in the second inequality we took advantage of (${\bf H}_3$), the third inequality is achievable by using Young's inequality, and the last inequality is valid owing to $\eta\in[0,1).$ Subsequently, \eqref{W2} follows by noticing  from  $r\le f(r)/{c_1},r\ge0,$ that
\begin{align*}
\Gamma_2^{\vv}(x,y)
  \le6L_V(1+L_V^\eta)\gamma^{1-\eta}(1+c_2/{c_1})\rho_{f,\gg,V}(x,y) \le \frac{c_3}{4(c_1+c_2)}\rho_{f,\gg,V}(x,y),
\end{align*}
where in the second inequality we leveraged the fact that
$
6L_V(1+L_V^\eta)\gamma^{1-\eta}(1+c_2/{c_1})\le   c_3/{4(c_1+c_2)}
$
 by invoking the definition of $\gamma$ defined in \eqref{Q10}.
\end{proof}

With Lemmas \ref{lem2} and \ref{lem3} at hand, we are in position to complete the proof of Theorem \ref{thm1}.

\begin{proof}[Proof of Theorem \ref{thm1}]
Below, we choose $(X^{\delta,\vv}_0,X^{(\delta,\vv)}_0)=(X^\nu_0,X^{(\delta),\mu}_0)$ such that
\begin{align*}
\mathbb W_{\rho_V}(\mu,\nu)=\E\big((1\wedge |X^\nu_0-X^{(\delta),\mu}_0|)(1+V(X^\nu_0)+V(X^{(\delta),\mu}_0))\big).
\end{align*}
By combining Lemma \ref{Lemma1} with Lemma \ref{lem2} and  Lemma \ref{lem3}, we derive that
\begin{equation*}
\begin{split}
\d \rho_{g,\gamma,V}( X^{\dd,\vv}_t, X^{(\delta,\vv)}_t)  &\le -\Big(\frac{c_3}{4(c_1+c_2)}\wedge\ff{\gg \ll}{  1+2\gg } \Big)\rho_{g,\gg,V}(X^{\dd,\vv}_t, X^{(\delta,\vv)}_t)\,\d t+R^{\delta,\vv,\gamma}_t\,\d t+\d M^{\delta,\vv,\gamma}_t\\
&\quad+2 \big( (c_3 + 2c_2  K_1)\vv +2K_2\vv^\alpha  \big)\big(1+\gamma V(X^{\dd,\vv}_t)+\gamma V ( X^{(\delta,\vv)}_t)\big)\,\d t .
\end{split}
\end{equation*}
 Next, notice from $({\bf H}_3)$ that for all $x\in\R^d,$
 \begin{align*}
 |\nn V(x)|
 &\le \int_0^1\|\nn^2V(sx)\|_{\rm op}\,\d s|x|+ |\nn V({\bf0})|\\
 &\le L_V\int_0^1(1+V(sx)^\eta)\,\d s|x|+ |\nn V({\bf0})|\\
 &\le 2L_V(1+L_V^\eta)(1+V({\bf0})^\eta)  V(x)^\eta   |x|+ |\nn V({\bf0})|\\
 &\le  L_V(1+L_V^\eta)(1+2L_V)(1+V({\bf0})^\eta) (1+V(x))+ |\nn V({\bf0})|.
 \end{align*}
 This, together with $\gamma\in(0,1]$, $g'_{-}(r)\le 2c_2$ and $g(r)\le 1+c_1l_0,r\ge0,$ leads to
\begin{align*}
R^{\delta,\vv,\gamma}_t \le (2c_2\vee c^\star)R^{\delta,\vv} _t,
\end{align*}
where $R_t^{\delta,\vv}$ is defined in \eqref{F3} and 
$$c^\star:=\gamma(1+c_1l_0)\big(L_V(1+L_V^\eta)(1+2L_V)(1+V({\bf0})^\eta)\vee |\nn V({\bf0})|\big).$$
Subsequently, we have
\begin{equation*}
\begin{split}
&\d \rho_{g,\gamma,V}( X^{ \delta,\vv }_t, X^{(\delta,\vv)}_t)  \\
&\le\d M^{\delta,\vv,\gamma}_t -\Big(\frac{c_3}{4(c_1+c_2)}\wedge\ff{\gg \ll_V}{  1+2\gg } \Big)\rho_{g,\gg,V}(X^{ \delta,\vv}_t, X^{(\delta,\vv)}_t)\,\d t+ (2c_2\vee c^\star)R^{\delta,\vv} _t\,\d t\\
&\quad+2 \big( (c_3 + 2c_2  K_1)\vv +2K_2\vv^\alpha  \big)\big(1+ V( X^{ \delta,\vv}_t)+ V ( X^{(\delta,\vv)}_t)\big)\,\d t .
\end{split}
\end{equation*}
Finally, by recalling that  $(X^{(\dd,\vv)}_t, X^{\dd,\vv }_t)_{t\ge0}$ is a   coupling of $(X^{(\delta),\mu}_t,X_t^\nu)_{t\ge0}$,
  \eqref{Q1} is derived from  Gronwall's inequality followed by applying  Lemma \ref{lem0} and  $({\bf H}_2)$ (so $\sup_{0\le s\le t}\E V(X_t^\nu)<\8$), sending $\vv\to0 $ and noting that
\begin{align*}
c_1\gamma (1\wedge|x-y|)\big(1+  V(x)+ V(y)\big)\le \rho_{g,\gamma,V}(x,y)\le l_0(c_1+c_2)(1\wedge|x-y| )\big(1+  V(x)+ V(y)\big).
\end{align*}
\end{proof}

\subsection{Proof of Theorem \ref{pro1}}

In order to apply Theorem \ref{thm1} to concrete EM type schemes, we need to treat  the  remainder term $\E R^{\delta,\vv}_t$  in \eqref{Q1}.
For this purpose, it is indispensable to provide an explicit criterion imposed on $b^{(\delta)}$ to examine  that $(X^{(\delta)}_t)_{t\ge0}$ has finite $p$-th moment in the infinite horizon. More precisely, besides $({\bf H}_4)$,  we
suppose that
\begin{enumerate}
	\item[$({\bf A}_0)$] there exist constants $c^\star,\ll^\star>0$ and $\delta_0\in(0,1]$  such that for all $x\in \R^d $ and $\delta\in(0,\delta_0],$
	\begin{align}
	\<x,b^{(\dd)}(x)\>+\delta|b^{(\dd)}(x)|^2 \le  -\ll^\star|x|^2+c^\star.	
	\end{align}

\end{enumerate}
\begin{lemma}\label{lem0-1}
Under $({\bf A}_0)$ and $({\bf H}_4)$, for $p> 0$, there exist constants $C_0,\lambda_0>0$ and $\dd_0\in(0,1)$ such that
for any $t>0$ and  $\delta\in(0,\delta_0]$,
\begin{align}\label{W8}
\E\big(|X^{(\delta)}_t|^p\big|\mathscr F_0\big)\le  \e^{-\lambda_0t}|X^{(\delta)}_0|^p  +C_0(1+d^{\ff p2}) .
\end{align}
\end{lemma}
\begin{proof}
Via Jensen's inequality, it is sufficient to show \eqref{W8} for  even integer  $p\ge 6$.
Applying It\^o's formula, we deduce from \eqref{eq2}, $({\bf A}_0)$ and $({\bf H}_4)$ that for all $\delta\in(0,\delta_0],$
\begin{align*}
\d |X^{(\delta)}_t|^p
&=p |X^{(\delta)}_t|^{p-2}\< X^{(\delta)}_{{t_\dd}},b^{(\dd)}(X^{(\delta)}_{t_\dd})\>\,\d t\\
&\quad+p |X^\dd_t|^{p-2}\<X^{(\delta)}_t-X^{(\delta)}_{t_\dd},b^{(\dd)}(X^{(\delta)}_{t_\dd})\>\,\d t+\ff{1}2p(d+p-2)|X^{(\delta)}_t|^{p-2}\,\d t +\d M_t\\
&\le-p\ll^\star  |X^{(\delta)}_t|^{p-2}
|X^{(\delta)}_{t_\dd}|^2\,\d t+p\dd^{\ff12} |X^{(\delta)}_t|^{p-2}|W^\dd_t|\cdot|b^{(\dd)}(X^{(\delta)}_{t_\dd})|\,\d t\\
&\quad+p|X^{(\delta)}_t|^{p-2}|b^{(\dd)}(X^{(\delta)}_{t_\dd})|^2\dd^2\,\d t+p\big(c^\star  + (d+p-2)/2\big) |X^{(\delta)}_t|^{p-2}\,\d t+\d M_t\\
&\le -p\ll^\star  |X^\dd_t|^{p-2}
|X^{(\delta)}_{t_\dd}|^2\,\d t+\frac{1}{2}p |X^{(\delta)}_t|^{p-2}\big(2( \lambda^\star)^{-1}(c^*)^2|W^\dd_t|^2\,\d t\\
&\quad+ \delta(\lambda^\star  (c^*)^{-2}+4\dd)(c_*^2+(c^*)^2\delta^{-2\theta}| X^{(\delta)}_{t_\dd} |^2\big)\big)\,\d t\\
&\quad+p\big(c^\star  + (d+p-2)/2\big) |X^{(\delta)}_t|^{p-2}\,\d t+\d M_t\\
&\le  -\frac{1}{2}p\ll^\star  |X^{(\delta)}_t|^{p-2}
|X^{(\delta)}_{t_\dd}|^2\,\d t+\d M_t\\
&\quad+p\big(c^\star  + (d+p-2)/2+( \lambda^\star)^{-1}(c^*)^2|W^\dd_t|^2\,\d t+\lambda^\star (c^*)^{-2} c_*^2/2+2\big) |X^{(\delta)}_t|^{p-2}\big)\,\d t,
 		\end{align*}
	where $(M_t)_{t\ge 0}$ is a local martingale, and $W^\delta_t:=(W_t-W_{t_\delta})(t-t_\delta)^{-\frac{1}{2}}\sim N({\bf 0}, I_d)$ due to the scaling property of  $(W_t)_{t\ge0}$. Again, we derive  from  \eqref{eq2}  that for all $\delta\in(0,1],$
\begin{align*}
|X^{(\delta)}_t|^2&=\big|X^{(\delta)}_{t_\delta}+b^{(\dd)}(X^{(\delta)}_{t_\dd})(t-t_\delta ) +W_t-W_{t_\delta}\big|^2\\
&\le3\big(|X^{(\delta)}_{t_\delta}|^2+|b^{(\dd)}(X^{(\delta)}_{t_\dd})|^2\delta^2+\dd|W^\dd_t|^2\big)\\
&\le3\big(|X^{(\delta)}_{t_\delta}|^2+2(c_*^2+(c^*)^2\delta^{-2\theta}|X^{(\delta)}_{t_\dd}|^2) \delta^2+\dd|W^\dd_t|^2\big)\\
&\le3\big((1+2(c^*)^2)|X^{(\delta)}_{t_\delta}|^2+2 c_*^2  \delta +\dd|W^\dd_t|^2\big)
\end{align*}
so that
\begin{align*}
\frac{1}{1+2(c^*)^2}\Big(\frac{1}{3}|X^{(\delta)}_t|^2-2 c_*^2\delta-\dd|W^\dd_t|^2\Big)\le|X^{(\delta)}_{t_\delta}|^2.
\end{align*}
This enables us to deduce that
\begin{align*}
\d |X^{(\delta)}_t|^p
&\le  - \frac{p\ll^\star}{6(1+2(c^*)^2)}   |X^{(\delta)}_t|^{p }
 \,\d t + \Gamma^\delta_t  |X^{(\delta)}_t|^{p-2}\,\d t+\d M_t,
\end{align*}
where
\begin{align*}
\Gamma^\delta_t:=&\frac{p\ll^\star}{2(1+2(c^*)^2)}\big(  2 c_*^2\delta+\dd|W^\dd_t|^2\big) \\
&+p\big(c^\star  + (d+p-2)/2+( \lambda^\star)^{-1}(c^*)^2|W^\dd_t|^2+\lambda^\star (c^*)^{-2} c_*^2/2+2\big).
\end{align*}
Next, by the Young inequality, it follows that for any $\vv>0,$
\begin{align*}
\Gamma^\delta_t  |X^{(\delta)}_t|^{p-2}\le \vv |X^{(\delta)}_t|^{p }+\frac{2}{p}\Big(\frac{p-2}{p\vv}\Big)^{\frac{p}{2}-1}(\Gamma^\delta_t)^{\frac{1}{2}p}.
\end{align*}
Whence, taking $\vv=\frac{p\ll^\star}{12(1+2(c^*)^2)}$ yields that for some constant $c_p>0,$
\begin{align*}
\d |X^{(\delta)}_t|^p
\le  - \frac{p\ll^\star}{12(1+2(c^*)^2)}   |X^{(\delta)}_t|^{p }
 \,\d t + c_p(1+d^{\frac{1}{2}p})\,\d t+\d M_t.
 \end{align*}
Consequently, the assertion \eqref{W8} is reachable by applying   Gronwall's inequality.
\end{proof}

\begin{lemma}\label{lem4}
For $p>0,$ let $V_p(x)=(1+|x|^2)^{\frac{1}{2}p}, x\in\R^d.$ Then,   under  $({\bf A}_2)$ and the hypothesis that for some  constants $L>0$ and $\alpha\in(0,1)$,
\begin{align}\label{EW8}
|b_0(x) |\le L(1+|x |^\alpha), \quad x\in\R^d,
\end{align}
$({\bf H}_2)$  is also satisfied so that for some constant $C_p>0,$
\begin{align}\label{EW13}
\E(|X_t|^p|\mathscr F_0)\le C_p(1+|X_0|^p)(d^{\ff p2 }+1).
\end{align}

\end{lemma}

\begin{proof}
Note that for all $x\in\R^d,$
\begin{align}\label{EW9}
\nn V_p(x)=p(1+|x|^2)^{\frac{p}{2}-1}x ~\mbox{ and }~ \nn^2 V_p(x)=p\Big(I_{d }+(p-2) \ff{x x^\top}{1+|x|^2}\Big)(1+|x|^2)^{\frac{p}{2}-1}.
\end{align}
Thus, one has
\begin{align*}
	(\mathscr LV_p)(x)\le p(1+|x|^2)^{\frac{p}{2}-1}\big(\<x,b(x)\>+   d/2+(p-2)^+/2 \big) .
\end{align*}
This, together with
\begin{align*}
\<x,b (x) \>&\le-\ll^*|x|^2+L|x|(1+ |x|^\aa)+C^* \le -\frac{1}{2}\ll^*(1+|x|^2)+ C_1
\end{align*}
by taking advantage of $({\bf A}_2)$ and \eqref{EW8}, implies that
\begin{align*}
	(\mathscr LV_p)(x)\le-\frac{1}{2} p\lambda^*V_p(x)+p(1+|x|^2)^{\frac{p}{2}-1}(C_1+   d/2+(p-2)^+/2  ),
\end{align*}
where
\begin{align*}
C_1:= C^*+\frac{1}{2}L(1-\alpha)\big(L(1+\alpha)/{\lambda^*}\big)^{\frac{1+\alpha}{1-\alpha}}+\frac{1}{2}\ll^*.
\end{align*}
Thus, $({\bf H}_2)$ is valid by applying Young's inequality. At length, the assertion \eqref{EW13} follows from Gronwall's inequality.
\end{proof}

By the aid of Theorem \ref{thm1},
Lemma \ref{lem0-1} as well as  Lemma \ref{lem4}, the following part is devoted to the proof of Theorem  \ref{pro1}.

\begin{proof}[Proof of Theorem \ref{pro1}]
In the  analysis below, we take $V(x)=(1+|x|^2)^{\frac{1}{2}p}, x\in\R^d,$ for $p>0.$ Once we can prove that there exist constants
$\ll^*,C_*, p^*,m^*>0,\dd^*\in (0,1)$ such that for all $\mu\in\mathscr	P_{p^*}(\R^d),\nu \in \mathscr P_{2p}(\R^d)$ and $\dd\in(0,\dd^*)$,
\begin{align}\label{JJ*}
	\mathcal W_{\rho_V}\big(\mathscr L_{X^{(\dd),\mu}_t},\mathscr L_{X^{\nu}_t}\big)\le C_*\e^{-\ll^* t}	\mathcal W_{\rho_{V}}(\mu,\nu)+C_*d^{m^*}\dd^{\ff\aa2},
	\end{align}
	the desired assertion \eqref{pro1*} can be attainable by noting the fact that
\begin{align*}
2^{-\ff2p}(1+|x|^p)\le V(x)\le 2^{(\frac{1}{2}p-1)^+}(1+|x|^p),\quad x\in\R^d.
\end{align*}
So, it boils down to verifying \eqref{JJ*} in order to complete the proof of Theorem \ref{pro1}.

According to  Lemma   \ref{lem4}, $({\bf H}_2)$ is satisfied. Subsequently, based on Remark \ref{Rem*} (ii),
 key points are left  to check  $({\bf H}_4)$ and $({\bf A}_0)$, and estimate the corresponding remainder by making use of
Theorem \ref{thm1}.

{(1) \it Proof of \eqref{JJ*} for    the   modified truncated EM scheme  \eqref{MTEM}.}
As regards the modified truncated EM scheme  \eqref{MTEM},  we stipulate
$$\delta_0=2^{-\frac{1}{\theta}}\wedge \bigg(\frac{\lambda^*}{8(2 \lambda^* +\alpha+  L_0 )^2 }\bigg)^{\frac{1}{1-2\theta}},$$
 and    choose
 \begin{align}\label{EW10}
 b^{(\delta)}(x) =-\lambda^*x+b_0(x)+\bar b_1(\pi^{(\delta,\theta)}(x)),\quad x\in\R^d.
 \end{align}
Then,  it follows from (${\bf A}_1$) that  for any  $x\in\R^d,$
\begin{align*}
 |b^{(\delta)}(x)-b(x)|&=\big|-\lambda^*(x -\pi^{(\delta,\theta)}(x))+b_1(\pi^{(\delta,\theta)}(x)) -b_1(x) \big|\\
 &\le \big(\lambda^*  +L_0(1+|x|^{\ell_0}+|\pi^{(\delta,\theta)}(x)|^{\ell_0})\big)|x -\pi^{(\delta,\theta)}(x)|\\
 &\le \big(\lambda^*  +L_0(1+2|x|^{\ell_0}) \big)|x -\pi^{(\delta,\theta)}(x)|\\
 &= \big(\lambda^*  +L_0(1+2|x|^{\ell_0}) \big)\big(|x|-\big(|x|\wedge(\dd^{-\theta}-1)^{\ff{1}{\ell_0}}\big)\big),
\end{align*}
where in the second inequality we used the fact that $|\pi^{(\delta,\theta)}(x)|\le|x|$ by taking the definition of $\pi^{(\delta,\theta)}$ into consideration, and in the identity we utilized
$
\big| x-\pi^{(\delta,\theta)}(x)\big|= |x|-\big(|x|\wedge(\dd^{-\theta}-1)^{\ff{1}{\ell_0}}\big) .
$
As a consequence, we infer that, for fixed $x\in\R^d,$  $|b^{(\delta)}(x)-b(x)|\to0$ as $\delta\to0.$

  By combining  \eqref{EW10} with
\begin{align}\label{EW15}
 |b_0(x)|\le |b_0({\bf0})|+K_2  |x|^\alpha ,\quad |\pi^{(\delta,\theta)}(x)|\le|x|,\quad |b_1(x)|\le|b_1({\bf0})|+L_0(1+|x|^{\ell_0})|x|,
 \end{align}
as well as $|\pi^{(\delta,\theta)}(x)|^{\ell_0}=|x|^{\ell_0}\wedge(\delta^{-\theta}-1)$,
 it is obvious that for all $x\in\R^d $ and $\delta\in(0,1],$
\begin{equation}\label{EW12}
\begin{split}
 |b^{(\delta)}(x)|&\le \lambda^*|x|+|b_0(x)|+\lambda^*| \pi^{(\delta,\theta)}(x) |+|b_1(\pi^{(\delta,\theta)}(x))| \\
 &\le 2\lambda^*|x|+K_2  |x|^\alpha+L_0(1+|x|^{\ell_0}\wedge(\delta^{-\theta}-1))|x|+|b_0({\bf0})|+  |b_1({\bf0})|\\
 &\le(2\lambda^*+\alpha+L_0\delta^{-\theta})|x|+c_1,
 \end{split}
 \end{equation}				
where $c_1:=(1-\alpha)K_2^{\frac{1}{1-\alpha}} +  |b_0({\bf0})|+  |b_1({\bf0})|$.
Thereby, $({\bf H}_4)$ is verifiable.

By invoking (${\bf A}_2$), along with $x=\frac{|x|}{|\pi^{(\delta,\theta)}(x)|}  \pi^{(\delta,\theta)}(x)$ for $x\in\R^d$ with $|x|\neq0$,
 we find that  for all $x\in\R^d$ with $|x|\neq 0,$
\begin{equation}\label{EYY}	
\begin{split}
 \<x, b^{(\delta)}(x)\> &=-\lambda^*|x|^2+\<x,b_0(x)\>+ \<x,b_1(\pi^{(\delta,\theta)}(x))+\ll^*\pi^{(\delta,\theta)}(x) \> \\
 &=-\lambda^*|x|^2+\<x,b_0(x)\> +\frac{|x|}{|\pi^{(\delta,\theta)}(x)|} \<\pi^{(\delta,\theta)}(x),b_1(\pi^{(\delta,\theta)}(x))+\ll^*\pi^{(\delta,\theta)}(x) \> \\
 &\le -\lambda^*|x|^2+(|b_0({\bf0})|\cdot|x|+K_2|x|^{1+\alpha})+\frac{C^*|x|}{|\pi^{(\delta,\theta)}(x)|},
\end{split}
\end{equation}
where in the last display we employed $|b_0(x)|\le |b_0({\bf0})|+K_2|x|^\alpha$ (see $({\bf H}_2)$ for more details).
Next, for any $x\in\R^d$ with $|x|\neq 0 $ and $\delta\in(0,\delta_0],$
\begin{align*}
\frac{ |x|}{|\pi^{(\delta,\theta)}(x)|}&=\frac{ |x|}{ |x|\wedge(\dd^{-\theta}-1)^{\ff{1}{\ell_0}} }\I_{\{|x|\le1\}}+\frac{ |x|}{ |x|\wedge(\dd^{-\theta}-1)^{\ff{1}{\ell_0}} }\I_{\{|x|>1\}}\\
&\le \I_{\{|x|\le1\}}+|x|\I_{\{|x|>1\}}\\
&\le  1+|x| .
\end{align*}
This, along with  \eqref{EW12} and \eqref{EYY}, further yields that for all $\delta\in(0,\delta_0]$ and $x\in\R^d,$
\begin{align*}	
 \<x, b^{(\delta)}(x)\>+\delta|b^{(\delta)}(x)|^2
 &\le -\lambda^*|x|^2+(|b_0({\bf0})|\cdot|x|+K_2|x|^{1+\alpha})+ C^*( 1+|x|)\\
 &\quad+2\big((2\lambda^*+\alpha+L_0\delta^{-\theta})^2|x|^2+c_1^2\big) \delta.
\end{align*}
Whence, due to $2 (2\lambda^*+\alpha+L_0\delta^{-\theta})^2  \delta\le \frac{1}{4}\lambda^*$ for $\delta\in(0,\delta_0]$,
 we conclude from  the Young inequality that there exists a constant $C_1>0$ such that  for all $\delta\in(0,\delta_0]$ and $x\in\R^d,$
\begin{align*}
\<x, b^{(\delta)}(x)\>+\delta|b^{(\delta)}(x)|^2 \le - \frac{1}{2}\lambda^*|x|^2+C_1
\end{align*}
so that  (${\bf A}_0$) is available.

 In terms of \eqref{EW10}, we obtain from $({\bf H}_1)$ and  $({\bf A}_1)$ that  for any $x,y\in\R^d$,
\begin{align*}
 |b(x)-b^{(\dd)}(y)|&=|b_0(x)+b_1(x)-(-\lambda^*y+b_0(y)+ \lambda^*  \pi^{(\delta,\theta)}(y) +b_1(\pi^{(\delta,\theta)}(y)))|\\
 &\le |b_0(x)-b_0(y)|+\lambda^*|y-\pi^{(\delta,\theta)}(y)|+|b_1(x)-b_1(y)|+|b_1(y)-b_1(\pi^{(\delta,\theta)}(y))|\\
 &\le K_2|x-y|^\alpha +L_0(1+|x|^{\ell_0}+|y|^{\ell_0})|x-y|\\
 &\quad+ \big(\lambda^*+L_0(1+ |y|^{\ell_0}+|\pi^{(\delta,\theta)}(y)|^{\ell_0})\big)|y-\pi^{(\delta,\theta)}(y)|.
 \end{align*}
 Furthermore,  one apparently has for any $\alpha>0,$
 \begin{equation}\label{EW11}
 \begin{split}
 |x-\pi^{(\delta,\theta)}(x)|&= \big(|x|- (\dd^{-\theta}-1)^{\ff{1}{\ell_0}} \big)\I_{\{(\dd^{-\theta}-1)^{\ff{1}{\ell_0}}\le|x|\}} \\
 &\le  \big(|x|- (\dd^{-\theta}-1)^{\ff{1}{\ell_0}} \big) |x|^\alpha (\dd^{-\theta}-1)^{-\ff{\alpha}{\ell_0}} \\
 &\le  |x|^{1+\alpha} (\dd^{-\theta}/2)^{-\ff{\alpha}{\ell_0}} ,
\end{split}
\end{equation}
 where in the second inequality we exploited $\dd^{-\theta}-1\ge\dd^{-\theta}/2$ for any $\delta\in(0,\delta_0].$ Thereafter, taking $\alpha=\ell_0/{(2\theta)}$ in \eqref{EW11} gives that for all $ x \in\R^d$ and $\delta\in(0,\delta_0]$
 \begin{equation*}
 |x-\pi^{(\delta,\theta)}(x)|
  \le 2^{\frac{1}{2\theta}}\delta^{\frac{1}{2}} |x|^{1+\ell_0/{(2\theta)}} .
\end{equation*}
 This, besides $|\pi^{(\delta,\theta)}(y)|\le |y|,$ leads to
 \begin{align*}
 |b(x)-b^{(\dd)}(y)|
 &\le K_2|x-y|^\alpha +L_0(1+|x|^{\ell_0}+|y|^{\ell_0})|x-y|\\
 &\quad+ 2^{\frac{1}{2\theta}}\delta^{\frac{1}{2}} \big(\lambda^*+L_0(1+2|y|^{\ell_0} )\big)|y|^{1+\ell_0/{(2\theta)}}, \quad x,y\in\R^d.
 \end{align*}
 As a result, we reach that
 \begin{equation}\label{F1}
 \begin{split}
 \E|R^{\dd,\vv}_t|&\le \E\Big(\Big(K_2|X^{(\dd,\vv)}_t-X^{(\dd,\vv)}_{t_\delta}|^\alpha +L_0(1+|X^{(\dd,\vv)}_t|^{\ell_0}+|X^{(\dd,\vv)}_{t_\delta}|^{\ell_0})|X^{(\dd,\vv)}_t-X^{(\dd,\vv)}_{t_\delta}|\\
 &\quad\quad\quad+2^{\frac{1}{2\theta}}\delta^{\frac{1}{2}} \big(\lambda^*+L_0(1+2|X^{(\dd,\vv)}_{t_\delta}|^{\ell_0} )\big)|X^{(\dd,\vv)}_{t_\delta}|^{1+\ell_0/{(2\theta)}}\Big)\\
 &\quad\quad\times \big(1+  (1+|X^{ \delta,\vv }_t|^2)^{\frac{1}{2}p}+(1+ |X^{(\delta,\vv)}_t|^2)^{\frac{1}{2}p}\big)\Big).
\end{split}
\end{equation}
Let 
\begin{align*}
\hat W_t=\int_0^th_\vv(|Z^{ \dd,\vv }_s|)\Pi(Z^{ \dd,\vv }_s)\, \d   \bar W_s+ \int_0^th^*_\vv(|Z^{ \dd,\vv }_s|)\,\d \tilde W_s,\quad t\ge0, 
\end{align*}
which is a standard $d$-dimensional Brownian motion. 
Next, from \eqref{eq3}, \eqref{EW10} and \eqref{EW12}, we have for any $\beta\in(0,1],$
\begin{align*}
 |X^{(\dd,\vv)}_t-X^{(\dd,\vv)}_{t_\dd}|^\beta&\le\big((2\lambda^*+\alpha+L_0 )|X^{(\dd,\vv)}_{t_\dd}|+c_1    +   |\hat W^\delta_t| \big)^\beta(t-t_\delta)^{\frac{1}{2}\beta}\\
 &\le \big((2\lambda^*+\alpha+L_0 )^\beta|X^{(\dd,\vv)}_{t_\dd}|^\beta+c_1^\beta    +   |\hat W^\delta_t|^\beta \big)\delta^{\frac{1}{2}\beta},
\end{align*}
where $\hat W^\delta_t:=(\hat W_t-\hat W_{t_\delta})(t-t_\delta)^{-\frac{1}{2}}\in N({\bf 0}, I_d)$.
Subsequently, by recalling that   $(X^{(\dd,\vv)}_t, X^{\dd,\vv }_t)_{t\ge0}$ is a   coupling of $(X^{(\delta),\mu}_t,X_t^\nu)_{t\ge0}$,
we deduce that there exist constants $c_2,c_3,c_4>0$ such that
\begin{align*}
 \E|R^{\dd,\vv}_t|&\le c_2 \delta^{\frac{1}{2}\alpha} \big(1+\E|X^{(\delta),\mu}_{t_\delta}|^{2\alpha} +\E|X^\nu_t|^{2p}+\E|X^{(\delta),\mu}_t|^{2p} \big) \\
 &\quad+c_3\delta^{\frac{1}{2} } \big(1+\E|X^{(\dd),\mu}_{t_\dd}|^{2(\ell_0+1)} +\E|X^{(\delta),\mu }_t|^{2\ell_0}+\E|X^\nu_t|^{2p}+\E|X^{(\delta),\mu}_t|^{2p}   \big) \\
 &\quad+c_4 \delta^{\frac{1}{2} } \big(1+\E|X^{(\dd),\mu}_{t_\dd}|^{2(\ell_0+1+\ell_0/{(2\theta)})}  +\E|X^\nu_t|^{2p}+\E|X^{(\delta),\mu}_t|^{2p}   \big).
\end{align*}
The previous estimate, in addition to Lemma \ref{lem0-1} and Lemma \ref{lem4}, enables us to finish the proof of Theorem  \ref{pro1} for    the modified      scheme \eqref{MTEM}.

{(2) \it  Proof of \eqref{JJ*} for    the   modified tamed  EM scheme \eqref{TEM}.}
In the following analysis, we take
\begin{align*}
\delta_0=1\wedge\bigg(\frac{\lambda^*}{6(4(\lambda^*)^2+ 2 L_0^2+(K_2\aa)^2) }\bigg)^{\frac{1}{1-2\theta}}
\end{align*}
and stipulate $\delta\in(0,\delta_0]$.  Regarding the  scheme \eqref{TEM}, we choose
\begin{align}\label{F4}
b^{(\delta)}(x)=b_0(x)-\ll^*x+\ff{b_1(x)+\ll^*x}{(1+\dd^{2\theta}|x|^{2\ell_0})^{\ff12}},\quad x\in\R^d.
\end{align}
By virtue of the basic  inequality: $(1+a)^{\frac{1}{2}}-1\le\frac{1}{2}a, a\ge0,$
it holds that for all $x\in\R^d,$
\begin{align*}
|b^{(\delta)}(x)-b(x)|
 =\frac{|b_1(x)+\ll^*x|((1+\dd^{2\theta}|x|^{2\ell_0})^{\ff12}-1)}{(1+\dd^{2\theta}|x|^{2\ell_0})^{\ff12}} \le \frac{|b_1(x)+\ll^*x||x|^{2\ell_0} \dd^{2\theta} }{2(1+\dd^{2\theta}|x|^{2\ell_0})^{\ff12}}.
\end{align*}
Therefore, we conclude that,    for fixed $x\in\R^d,$ $|b^{(\delta)}(x)-b(x)| \to0$ as $\delta\to0. $

By means of the fact that  for $\delta\in(0,1],$
$$
\frac{ 1+r }{(1+\delta^{2\theta }r^2)^{\frac{1}{2}}}\le \ss2\delta^{- \theta}, \quad r\ge0,
$$
and \eqref{EW15}, we derive that for all $x\in\R^d,$
\begin{equation}\label{EW16}
\begin{split}
|b^{(\delta)}(x)|&\le |b_0(x)|+2\lambda^*|x|+\ff{|b_1(x) |}{(1+\dd^{2\theta}|x|^{2\ell_0})^{\ff12}}\\
&\le |b_0({\bf0})|+|b_1({\bf0})|+K_2  |x|^\alpha+2\lambda^*|x|+\ff{ L_0(1+|x|^{\ell_0})|x|}{(1+\dd^{2\theta}|x|^{2\ell_0})^{\ff12}}\\
&\le |b_0({\bf0})|+|b_1({\bf0})|+K_2(1-\alpha)   +(K_2\alpha  +2\lambda^*+\ss2L_0\delta^{- \theta})|x|.
\end{split}
\end{equation}
As a consequence, (${\bf H}_4$) is valid by making use of Young's inequality.
Next,   we obtain from (${\bf A}_2$), \eqref{EW16} as well as $|b_0(x)|\le |b_0({\bf0})|+K_2  |x|^\alpha$    that for all $\delta\in(0,\delta_0],$
\begin{align*}
 \<x, b^{(\delta)}(x)\>+\delta|b^{(\delta)}(x)|^2&=- \ll^*|x|^2+\<x,b_0(x)\>+\ff{\<x,b_1(x)+\ll^*x\>}{(1+\dd^{2\theta}|x|^{2\ell_0})^{\ff12}} +\delta|b^{(\delta)}(x)|^2\\
 &\le  - \ll^*|x|^2+3(4(\lambda^*)^2+  2L_0^2+(K_2\aa)^2) \delta^{1- 2\theta}|x|^2\\
 &\quad+(|b_0({\bf0})|+K_2  |x|^\alpha)|x|+2(|b_0({\bf0})|+|b_1({\bf0})|+K_2 (1-\aa))^2 +C^*\\
 &\le - \frac{1}{2}\ll^*|x|^2+(|b_0({\bf0})|+K_2  |x|^\alpha)|x|\\
 &\quad+2(|b_0({\bf0})|+|b_1({\bf0})|+K_2  (1-\aa))^2 +C^*,
\end{align*}
 where in the last inequality we utilized the fact that $3(4(\lambda^*)^2+  2L_0^2+(K_1\aa)^2) \delta^{1- 2\theta}\le \frac{1}{2}\lambda^*$ for $\delta\in(0,\delta_0]$. Accordingly,
 (${\bf A}_0$) is examinable.

Next, by invoking $({\bf H}_1)$, $({\bf A}_1)$, as well as the inequality: $(1+a)^{\frac{1}{2}}-1\le\frac{1}{2}a, a\ge0,$ we infer that for all $x,y\in\R^d,$
\begin{align*}
 |b(x)-b^{(\dd)}(y)| &=\bigg|b_0(x) -b_0(y)+b_1(x)+\ll^*y-\ff{b_1(y)+\ll^*y}{(1+\dd^{2\theta}|y|^{2\ell_0})^{\ff12}}\bigg|\\
 &\le K_2|x-y|^\alpha+|b_1(x)-b_1(y)|+\frac{|b_1(y)+\ll^*y|((1+\dd^{2\theta}|y|^{2\ell_0})^{\ff12}-1)}{(1+\dd^{2\theta}|y|^{2\ell_0})^{\ff12}}\\
 &\le K_2|x-y|^\alpha+L_0(1+|x|^{\ell_0}+|y|^{\ell_0})|x-y|\\
 &\quad+\frac{1}{2}|y|^{2\ell_0}\big(|b_1({\bf0})|+(\ll^*+L_0(1+|y|^{\ell_0}))|y|\big) \dd^{2\theta}.
\end{align*}
Furthermore, it follows from \eqref{eq3}, \eqref{F4} and \eqref{EW16} that for any $\beta\in(0,1],$
\begin{align*}
 |X^{(\dd,\vv)}_t-X^{(\dd,\vv)}_{t_\dd}|^\beta&\le\big( c_0 +\dd^{\ff12}(K_2\alpha  +2\lambda^*+\ss2L_0\delta^{- \theta})|X^{(\dd,\vv)}_{t_\dd}| +   |\hat W^\delta_t| \big)^\beta (t-t_\delta)^{\frac{1}{2}\beta},
\end{align*}
where $c_0:=|b_0({\bf0})|+|b_1({\bf0})|+K_2(1-\alpha) $ and $\hat W^\delta_t:=(\hat W_t-\hat W_{t_\delta})(t-t_\delta)^{-\frac{1}{2}}$.
Correspondingly, we deduce  that there exist constants $c_5,c_6,c_7>0$ such that
\begin{equation}\label{F2}
\begin{split}
 \E|R^{\dd,\vv}_t|&\le \E\Big(\Big(K_2|X^{(\dd,\vv)}_t-X^{(\dd,\vv)}_{t_\delta}|^\alpha +L_0(1+|X^{(\dd,\vv)}_t|^{\ell_0}+|X^{(\dd,\vv)}_{t_\delta}|^{\ell_0})|X^{(\dd,\vv)}_t-X^{(\dd,\vv)}_{t_\delta}|\\
 &\quad\quad\quad+\frac{1}{2}\dd^{2\theta}|X^{(\dd,\vv)}_{t_\delta}|^{2\ell_0}\big(|b_1({\bf0})|+L_1(\ll^*+1+|X^{(\dd,\vv)}_{t_\delta}|^{\ell_0})
 |X^{\dd,\mu}_{t_\delta}|\big)\Big) \\
 &\quad\quad\times \big(1+  (1+|X^{\dd,\vv}_t|^2)^{\frac{1}{2}p}+(1+ |X^{(\dd,\vv)}_t|^2)^{\frac{1}{2}p}\big)\Big)\\
&\le c_5 \delta^{\frac{1}{2}\alpha} \big(1+\E|X^{(\delta),\mu}_{t_\delta}|^{2\alpha} +\E|X^\nu_t|^{2p}+\E|X^{(\delta),\mu}_t|^{2p} \big) \\
 &\quad+c_6\delta^{\frac{1}{2} } \big(1+\E|X^{(\dd),\mu}_{t_\dd}|^{\ell_0+1} +\E|X^{(\delta),\mu }_t|^{\ell_0+1}+\E|X^\nu_t|^{2p}+\E|X^{(\delta),\mu}_t|^{2p}   \big) \\
 &\quad+c_7 \delta^{2\theta } \big(1+\E|X^{(\dd),\mu}_{t_\dd}|^{3\ell_0+1}  +\E|X^\nu_t|^{2p}+\E|X^{(\delta),\mu}_t|^{2p}   \big),
\end{split}
\end{equation}
where in the second inequality we utilize the fact that $(X^{(\dd,\vv)}_t, X^{\dd,\vv }_t)_{t\ge0}$ is a   coupling of $(X^{(\delta),\mu}_t,X_t^\nu)_{t\ge0}$.

On the basis of the previous estimates, Lemma \ref{lem0-1} as well as  Lemma \ref{lem4}, the proof of \eqref{JJ*} for    the    scheme \eqref{TEM} can be finished by noticing $\theta\in[\frac{1}{4},\frac{1}{2})$.
\end{proof}

\section{Proofs of Theorem \ref{thm2} and Theorem \ref{prolem}}\label{sec3}
By following the line to finish the proof of Theorem \ref{thm1},  we first of all design a suitable coupling process between
$(X_t,Y_t)_{t\ge0}$ and $(X_t^{(\delta)},Y_t^{(\delta)})_{t\ge0}$.  For this purpose, we consider the following coupled SDE: for all $t>0$ and $\delta\in(0,1],$
\begin{align}\label{EQ3}
	\begin{cases}
 \d   X^{\delta,\vv}_t=(a   X^{\delta,\vv}_t+b  Y^{\delta,\vv}_t)\,\d t,  \\
 \d   Y^{\delta,\vv}_t=U( X^{\delta,\vv}_t, Y^{\delta,\vv}_t)\,\d t+ h_\vv(|Q^{\dd,\vv}_t|)\d \hat W_t+ h^*_\vv(|Q^{\dd,\vv}_t|)\,\d \tt W_t;\\
 \d   X^{(\dd,\vv)}_t=(a   X^{(\dd,\vv)}_{t_\dd}+b Y^{(\dd,\vv)}_{t_\dd})\,\d t , \\
 \d   Y^{(\dd,\vv)}_t=U(  X^{(\dd,\vv)}_{t_\dd}, Y^{(\dd,\vv)}_{t_
 \dd})\,\d t+ h_\vv(|Q^{\dd,\vv}_t|)\Pi(Q^{\dd,\vv}_t) \d \hat W_t+ h^*_\vv(|Q^{\dd,\vv}_t|)\,\d \tt W_t
\end{cases}
\end{align}
with the initial value $((   X^{\delta,\vv}_0,   Y^{\delta,\vv}_0),(   X^{(\dd,\vv)}_0,   Y^{(\dd,\vv)}_0))=((  X_0,  Y_0),(  X^{(\dd)}_0,  Y^{(\dd)}_0))$,
where $(\hat W_t)_{t\ge 0}$ and  $(\tt  W_t)_{t\ge 0}$ are mutually independent Brownian motions supported on the same filtered probability space as that of $(W_t)_{t\ge0}$;  $\Pi, h_\vv$ and $h_\vv^*$  are defined exactly as in \eqref{eq3}; for  $\gamma>0$ (which is to be fixed later)
\begin{align}\label{U2}
Q^{\dd,\vv}_t:=Z^{\dd,\vv}_t+\gamma V^{\dd,\vv}_t\quad \mbox{ with } \quad Z^{\dd,\vv}_t:=  X^{\delta,\vv}_t-  X^{(\dd,\vv)}_t \mbox{ and }   V^{\dd,\vv}_t:= Y^{\delta,\vv}_t- Y^{(\dd,\vv)}_t.
\end{align}
By L\'evy's  characterization of Brownian Motions,  $(( X_t^{\dd,\vv}, Y_t^{\delta,\vv}), ( X^{(\dd,\vv)}_t,  Y^{(\dd,\vv)}_t))_{t\ge0}$
 is a coupling process of $(X_t,Y_t)_{t\ge0}$ and $(X_t^{(\delta)},Y_t^{(\delta)})_{t\ge0}$; see  \cite[Lemma 2.1]{BH2} for related details.

\subsection{Proof  of Theorem \ref{thm2} }
Analogously to the proof of Theorem \ref{thm1}, we prepare for several corresponding  lemmas to carry out the proof of Theorem \ref{thm2}.
Below, for  $\aa_0>0$ (to be given later on), we set for any $t\ge0,$
 $$r^{\dd,\vv}_t:=\aa_0 |Z^{\dd,\vv}_t|+|Q^{\dd,\vv}_t|,$$
 and
\begin{align}\label{F11}
\Psi^{\dd,\vv}_t :=(a(\aa_0+1)+L\gamma)\big|  X^{(\dd,\vv)}_t- X^{(\dd,\vv)}_{t_\dd}\big|+(b(\aa_0+1)+L\gamma)\big|   Y^{(\dd,\vv)}_t-   Y^{(\dd,\vv)}_{t_\dd}\big|.
	\end{align}

\begin{lemma}\label{lem7}
Assume that  $({\bf B}_1)$ holds true. Then, for any $\gamma\in(0, b/a] $ and $t\ge0,$
\begin{equation}\label{U3}
\begin{split}
		\d r^{\dd,\vv}_t&\le
\big(((1-\aa_0)(b\gamma^{-1}-a) + L(1+\gamma))|Z^{\dd,\vv}_t|	+((\aa_0+1)b\gamma^{-1}+L)|Q^{\dd,\vv}_t|+\Psi^{\dd,\vv}_t\big)\,\d t \\
		&\quad+2\gamma h_\vv(|Q^{\dd,\vv}_t|)\mathds 1_{\{ |Q	^\dd_t|\neq 0\}}\big< {\bf n}(Q	^{\dd,\vv}_t), \d \hat W_t\big>.
\end{split}
\end{equation}
\end{lemma}

\begin{proof}
It follows from \eqref{U2} that
\begin{equation}\label{U8}
\begin{split}
\d Z^{\dd,\vv}_t&=\big(a( X^{\dd, \vv}_t- X^{(\dd,\vv)}_{t_\dd})+b(Y^{\dd,\vv}_t- Y^{(\dd,\vv)}_{t_\dd})\big)\d t\\
&=\big(a  Z^{\dd,\vv}_t+b V^{\dd,\vv}_t+a( X^{(\dd,\vv)}_t-  X^{(\dd,\vv)}_{t_\dd})+b( Y^{(\dd,\vv)}_t-  Y^{(\dd,\vv)}_{t_\dd})\big)\d t\\
&=\big((a -b\gamma^{-1} ) Z^{\dd,\vv}_t+b\gamma^{-1}  Q^{\dd,\vv}_t +a( X^{(\dd,\vv)}_t-  X^{(\dd,\vv)}_{t_\dd})+b( Y^{(\dd,\vv)}_t- Y^{(\dd,\vv)}_{t_\dd})\big)\d t.
\end{split}
\end{equation}
With this in hand, we arrive at
	\begin{align}\label{b3}
		\d |Z^{\dd,\vv}_t| \le \big((a-b\gamma^{-1}) |Z^{\dd,\vv}_t|	+b\gamma^{-1} |Q^{\dd,\vv}_t|\big)\,\d t+\big(a\big|    X^{(\dd,\vv)}_t-    X^{(\dd,\vv)}_{t_\dd}\big|+b\big|   Y^{(\dd,\vv)}_t-   Y^{(\dd,\vv)}_{t_\dd}\big|\big)\,\d t.
		\end{align}
Notice from \eqref{EQ3} and \eqref{U8} that
\begin{align*}
	\d Q^{\dd,\vv}_t
	   &= \big((a -b\gamma^{-1} ) Z^{\dd,\vv}_t+b\gamma^{-1}  Q^{\dd,\vv}_t +a(  X^{(\dd,\vv)}_t-  X^{(\dd,\vv)}_{t_\dd})+b( Y^{(\dd,\vv)}_t- Y^{(\dd,\vv)}_{t_\dd})\big) \,\d t\\
	   &\quad+\gamma\big(U(    X^{\dd,\vv}_t,  Y^{\dd,\vv}_t)- U(     X^{(\dd,\vv)}_{t_\dd},   Y^{(\dd,\vv)}_{t_\dd})\big)\,\d t+2\gamma   h_\vv(|Q^{\dd,\vv}_t|){\bf n}(Q^{\dd,\vv}_t)\otimes{\bf n}(Q^{\dd,\vv}_t )\d \hat W_t.
\end{align*}
Then, by repeating the procedure to derive \eqref{Q2}, we deduce from (${\bf B}_1$)  that
\begin{equation}\label{U4}
\begin{split}
		\d |Q	^{\dd,\vv}_t|
&\le \big( ( b\gamma^{-1} -a +	L\gamma  +L) |Z^{\dd,\vv}_t|+(b\gamma^{-1}+  L) | Q^{\dd,\vv}_t| \big)\,\d t \\
&\quad + \big((a+L\gamma)|  X^{(\dd,\vv)}_t-  X^{(\dd,\vv)}_{t_\dd}|+(b+L\gamma)|  Y^{(\dd,\vv)}_t-  Y^{(\dd,\vv)}_{t_\dd}| \big)\,\d t\\
&\quad +2\gamma h_\vv(|Q^{\dd,\vv}_t|)\mathds 1_{\{ |Q	^{\dd,\vv}_t|\neq 0\}}\<{\bf n}(Q^{\dd,\vv}_t ), \d \hat W_t\>.
	\end{split}
\end{equation}
Accordingly, the assertion \eqref{U3} follows by combining \eqref{b3} with \eqref{U4}.
\end{proof}

In the   analysis below, let $g:[0,\infty)\to[0,\infty)$ be a concave function, and write $z =x-x', v=y-y',
q =z+\gamma v$ and $r=r((x,y),(x',y')) =\aa_0|z|+|q|$. For    $(x,y),(x',y')\in\R^{2d}$ and $\beta>0,$ we set
\begin{equation}\label{U16}
\begin{split}
\Theta^\vv_1((x,y),(x',y'))&:= \big(g'_{-}(r)\big(((1-\aa_0)( b\gamma^{-1}-a) + L(1+\gamma))|z|	+((\aa_0+1)b\gamma^{-1}+L)|q|\big)\\
&\quad\quad+2 \gamma^{2} g''(r)h_\vv(|q|)^2\big)\big(1+\beta V(x,y)+\beta V(x',y')\big) \\
&\quad-\beta g(r)\big( \ll_V^*(V(x,y)+V(x',y'))-2C_V^*\big),\\
\Theta^\vv_2((x,y),(x',y')):&=2\beta\gamma g'_{-}( r)h^2_\vv(|q|) \mathds 1_{\{ |q|\neq 0\}}\<{\bf n}(q),  \nn_2  V( x',  y')-\nn_2  V(  x,  y)\>,
\end{split}
\end{equation}
and
\begin{align*}
\rho_{V,g,\bb}((x,y),(x',y')): =g(r)\big(1+\bb V(x,y)+\bb V(x',y')\big).
\end{align*}

\begin{lemma}\label{lem8}
Assume   that $({\bf B}_1)$ and $({\bf B}_2)$ hold. Then, for any $\delta,\vv,\bb,t>0,$
\begin{equation}\label{W11}
	\begin{aligned}
 &\d \rho_{V,g,\bb}((   X^{\dd,\vv}_t,   Y^{\dd,\vv}_t),( X^{(\dd,\vv)}_t,  Y^{(\dd,\vv)}_t))\\
&\le  (\Theta^\vv_1+\Theta^\vv_2)\big((  X^{\dd,\vv}_t,  Y^{\dd,\vv}_t),( X^{(\dd,\vv)}_t,   Y^{(\dd,\vv)}_t)\big) \,\d t\\
&\quad +  \big( 1+\bb V(    X^{\dd,\vv}_t,   Y^{\dd,\vv}_t)+\bb V(  X^{(\dd,\vv)}_t, Y^{(\dd,\vv)}_t)\big)g'_{-}( r^{\dd,\vv}_t)\Psi^{\dd,\vv}_t \,\d t +\beta g(r^{\dd,\vv}_t)\Phi^{\dd,\vv}_t\,\d t+\d M^{\dd,\vv}_t  ,
\end{aligned}
\end{equation}
 in which  $(M^{\dd,\vv}_t)_{t\ge0}$
 is a martingale, and
  \begin{equation}\label{F12}
  \begin{split}
		\Phi^{\dd,\vv}_t&:=\big\<\nn_1V(  X^{(\dd,\vv)}_t,  Y^{(\dd,\vv)}_t),a( X^{(\dd,\vv)}_t-   X^{(\dd,\vv)}_{t_\dd})+b(   Y^{(\dd,\vv)}_t-  Y^{(\dd,\vv)}_{t_\dd})\big\>\\
		&\quad +\big\< \nn_2 V( X^{(\dd,\vv)}_t ,   Y^{(\dd,\vv)}_t), U(  X^{(\dd,\vv)}_t,   Y^{(\dd,\vv)}_t)-U(   X^{(\dd,\vv)}_{t_\dd},  Y^{(\dd,\vv)}_{t_\dd})\big\>.
	\end{split}
\end{equation}
\end{lemma}

\begin{proof}
The derivation of \eqref{W11} is similar to that of \eqref{Q6} so we herein give merely a sketch. As done in \eqref{U7},
applying It\^o's Tanaka formula enables us to obtain from \eqref{U3} that
	\begin{align*}
 &\d g(r^{\dd,\vv}_t)\\
		&\le   g'_{-}( r^{\dd,\vv}_t)\big((\aa_0(a-b\gamma^{-1}) + b\gamma^{-1}-a + L\gamma+L)|Z^{\dd,\vv}_t|	+((\aa_0+1)b\gamma^{-1}+L)|Q^{\dd,\vv}_t|+\Psi^{\dd,\vv}_t\big)\,\d t\\
 &\quad +2\gamma^{2} g''( r^{\dd,\vv}_t)h_\vv(|Q^{\dd,\vv}_t|)^2\,\d t +2\gamma g'_{-}( r^{\dd,\vv}_t)h_\vv(|Q^{\dd,\vv}_t|)\mathds 1_{\{ |Q	^{\dd,\vv}_t|\neq 0\}}\big<{\bf n}(Q	^{\dd,\vv}_t), \d \hat W_t\big>.
	\end{align*}
Furthermore, employing It\^o's formula  and taking (${\bf B}_2$) into consideration yield that
\begin{align*}	
\d V (  X^{\dd, \vv}_t,    Y^{\dd, \vv}_t)\le & \big(-\ll_V^* V (  X^{\dd, \vv}_t,   Y^{\dd, \vv}_t)+C_V^*\big)\,\d t \\ &+\big\< \nn_2 V(  X^{ \dd,\vv}_t,   Y^{\dd, \vv}_t),h_\vv(|Q^{\dd,\vv}_t|)\d \hat W_t+ h^*_\vv(|Q^{\dd,\vv}_t|)\,\d \tt W_t\>
\end{align*}
and
\begin{align*}
		\d V (  X^{(\dd,\vv)}_t,  Y^{(\dd,\vv)}_t)&\le \big(-\ll_V^* V (  X^{(\dd,\vv)}_t,   Y^{(\dd,\vv)}_t)+C_V^*\big)\,\d t+\Phi^{\dd,\vv}_t\,\d t\\
		&\quad+\big\< \nn_2  V(   X^{(\dd,\vv)}_t,  Y^{(\dd,\vv)}_t),  h_\vv(|Q^{\dd,\vv}_t|)\Pi(Q^{\dd,\vv}_t) \d \hat W_t + h^*_\vv(|Q^{\dd,\vv}_t|)\,\d \tt W_t\big\>.
	\end{align*}
Thus,  \eqref{W11} is achieved by the chain rule and the fact that
	\begin{align*}
	&	\d \big\<g( r^{\dd,\vv}_\cdot),V( X^{\dd,\vv}_\cdot,  Y^{\dd,\vv}_\cdot)+V(  X^{(\dd,\vv)}_\cdot,  Y^{(\dd,\vv)}_\cdot)\big>_t\\
		&=2\gamma g'( r^{\dd,\vv}_t)h^2_\vv(|Q^{\dd,\vv}_t|) \mathds 1_{\{ |Q	^{\dd,\vv}_t|\neq 0\}}\big<{\bf n}(Q	^{\dd,\vv}_t), {\Pi}(Q	^{\dd,\vv}_t)\nn_2  V(  X^{(\dd,\vv)}_t,  Y^{(\dd,\vv)}_t)+\nn_2  V(  X^{\dd,\vv}_t,    Y^{\dd,\vv}_t)\big>\\
&=2\gamma g'( r^{\dd,\vv}_t)h^2_\vv(|Q^{\dd,\vv}_t|) \mathds 1_{\{ |Q	^{\dd,\vv}_t|\neq 0\}}\big<{\bf n}(Q	^{\dd,\vv}_t),  \nn_2  V(   X^{\dd,\vv}_t,  Y^{\dd,\vv}_t)-\nn_2  V(  X^{(\dd,\vv)}_t,  Y^{(\dd,\vv)}_t)\big>,
	\end{align*}
where the second identity holds true due to \eqref{U9}.
\end{proof}

Set
\begin{align*}
 \mathcal D_V:= \big\{ ((x,y),(x',y'))\in\R^{2d}\times\R^{2d}: V(x,y)+V(x',y')\le {4 C_V^*}/{\ll_V^*}
  	\big\}.
  	\end{align*}
 Due to the  compact property of $V$, one has
 \begin{align*}
  \ell_0:=1+\sup_{((x,y),(x',y'))\in  \mathcal D_V}r((x,y),(x',y'))<\infty.
 \end{align*}
Note that the parameters $\gamma,\aa_0,\beta$ mentioned previously are free. In the sequel, we shall stipulate
\begin{align}\label{U12}
\gamma=\ff{b}{a+b}, \quad  \quad \aa_0=2+\Big(\frac{1}{b}+\frac{1}{a+b}\Big)L,
\end{align}
and
\begin{equation}\label{U18}
\begin{split}
\bb= 1\wedge&	\ff{b }{8C_V^*(\aa_0+\kk_0^{-1})}   \wedge
 \frac{c_1c_2\gamma^2}{4C_V^*\ell_0(1+\aa_0\kk_0)(c_1+c_2)}\\
 &\wedge\bigg(\frac{\lambda_*}{4L^*_V(1+c_2/{c_1})
\big(1\vee (({1+\gamma})/{\aa_0})\big) }\bigg)^{\frac{1}{1-\eta}},
\end{split}
\end{equation}
 where
\begin{align}\label{U14}
 \kk_0:=2((1+\aa_0)\gamma^{-1}+L/b),  ~   c_1:=   c_2\e^{-c_2 \ell_0} ~   \mbox{ with }~     c_2:= \ff{2}{\gamma ^2}\ell_0 ((\aa_0+1)b\gamma^{-1}+  L),
\end{align}
and
\begin{align*}
\lambda_*:=\frac{c_1c_2  \gamma^{2}   }{2\ell_0(\aa_0\kk_0+1)(c_1+c_2)}\wedge \ff{b}{4(\aa_0+\kk_0^{-1})}.
\end{align*}
Furthermore, for the sake of simplicity, we write
\begin{align}\label{U17}
C_*=4\big( (c_1+c_2)  ((\aa_0+1)b\gamma^{-1}+L ) +c_1c_2       \gamma^{2}/{\ell_0}\big).
\end{align}
In addition, for the function $g$ involved in \eqref{U16}, 	 we
choose  $g(r) =f(r\wedge\ell_0)$, $r\ge0,$ where
$
	f(r):=1-\e^{-c_2 r}+c_1 r
$
with constants $c_1,c_2>0$	being given in \eqref{U14}.

\begin{lemma}\label{lem9} For any $x,y,x',y'\in \R^d,$
\begin{equation}\label{U11}
\begin{split}
		\Theta^\vv_1((x,y),(x',y'))
	&	\le  -\bigg(\lambda_*\I_{\{r\le\ell_0\}}+\ff {2	\bb C_V^*}{1+4\bb C_V^*/\ll^*_V}\I_{\{r>\ell_0 \}} \bigg)
	\rho_{V,g,\bb}(( x,y),(  x',y'))\\
&\quad +C_*\vv (1+\beta V(x,y)+\beta V(x',y')),
\end{split}
\end{equation}
and
\begin{align}\label{U10}
\Theta^\vv_2((x,y),(x',y'))\le 2L^*_V(1+c_2/{c_1})\beta^{1-\eta}
\big(1\vee (({1+\gamma})/{\aa_0})\big)  \rho_{V,g,\beta}((x,y),(x',y'))\mathds 1_{\{r\le \ell_0\}}.
\end{align}

\end{lemma}

\begin{proof}
 In the first place, we prove \eqref{U10}.
 \eqref{a2}, along with Young's inequality, implies that
		\begin{align*}
		 |\nn _2 V(x,y)-\nn_2 V(x',y')|
		&\le L^*_V( V(x,y)^\eta+ V(x',y')^\eta)(|z|+|v|)
		\\
	 &\le   L^*_V\bb^{-\eta}\big(({1+\gamma^{-1}})/{\aa_0} )\vee \gamma^{-1}\big)\big(2+ \bb
 V(x,y)+\bb  V(x',y')   \big) r  \\
&\le  2L^*_V\beta^{-\eta} \big( (({1+\gamma^{-1}})/{\aa_0} )\vee \gamma^{-1}\big)(1+\beta V(x,y)+\beta V(x',y') )r,
	\end{align*}
	where in the second inequality we employ the fact that
$$|z|+|v|=|z|+\gamma^{-1}|q-z|\le \big((1+\gamma^{-1})/{\aa_0}\vee \gamma^{-1}\big)r.$$
Correspondingly, \eqref{U10} follows by means of $h_\vv\in[0,1]$, $f'(r)\le c_1+c_2$ as well as $r\le f(r)/{c_1}$.

We proceed to verify \eqref{U11} on account of the cases $(i)$ $r>\ell_0 $,  $(ii)$ $r\le \ell_0 $ and $|z|\ge \kk_0|q|$, and $(iii)$ $r\le \ell_0 $ and $|z|<\kk_0|q|$.

For the case $(i)$, which obviously implies that    $V(x,y)+V(x',y')> {4 C_V^*}/{\ll_V^*}$,
$g'(r)=g''(r)=0$ yields that
\begin{align*}
\Theta^\vv_1((x,y),(x',y'))
&= -\ff{\beta\big(\ll_V^*(V(x,y)+V(x',y'))-2C_V^*\big)}{1+\beta(V(x,y)+V(x',y'))}\rho_{V,g,\bb}(( x,y),(  x',y'))\\
&\le -\ff {2	\bb C_V^*}{1+4\bb C_V^*/\ll^*_V}  \rho_{V,g,\bb}(( x,y),(  x',y')),
\end{align*}
where the inequality is valid by noticing that
 	for $u\ge {4 C_V^*}/{\ll_V^*},$\begin{align*}
		\ff{\bb(\ll^*_V u-2C_V^*)}{1+\bb u}\ge\ff {2	\bb C_V^*}{1+4\bb C_V^*/\ll^*_V} .
	\end{align*}

Concerning the case $(ii)$, due to $g''(r)\le0,$ it follows that
\begin{align*}
\Theta^\vv_1((x,y),(x',y'))& \le g'_{-}(r)\big((1-\aa_0)( b\gamma^{-1}-a) + L(1+\gamma)	+\kk_0^{-1}((\aa_0+1)b\gamma^{-1}+L)\big)|z|\\
&\quad\times \big(1+\beta V(x,y)+\beta V(x',y')\big)+2\beta  C_V^*  g(r).
\end{align*}
With the alternatives of $\gamma,\alpha_0$ given in \eqref{U12}, one has
\begin{align}\label{U13}
(1-\aa_0)( b\gamma^{-1}-a) + L(1+\gamma)=-b.
\end{align}
This, besides $\kk_0^{-1}((\aa_0+1)b\gamma^{-1}+L)= b/2 $ and $  \ff{|z|}{r}\ge \ff{1}{\aa_0+\kk_0^{-1}}$ (recalling $|z|> \kk_0|q|$),
 leads to
\begin{align*}
\Theta^\vv_1((x,y),(x',y'))&\le - \ff{b}{2(\aa_0+\kk_0^{-1})} rg'(r)(1+\beta(V(x,y)+V(x',y')))+2\beta C_V^* g(r)\\
&\le - \ff{b}{2(\aa_0+\kk_0^{-1})}  \rho_{V,g,\bb}(( x,y),(  x',y'))+2\beta C_V^* g(r)\\
&\le - \ff{b}{4(\aa_0+\kk_0^{-1})}  \rho_{V,g,\bb}(( x,y),(  x',y')),
\end{align*}
where in the second inequality we make use of $g(r)\le rg'(r)$ for $ 0\le r\le \ell_0,$ and $2\beta C_V^*\le \ff{b}{4(\aa_0+\kk_0^{-1})}$.

 With regarding to the case $(iii)$,
 we deduce from \eqref{U13} that
\begin{align*}
\Theta^\vv_1((x,y),(x',y'))& \le  \big(f'_{-}(r)  ((\aa_0+1)b\gamma^{-1}+L)|q|+2 \gamma^{2} f''(r)h_\vv(|q|)^2\big)\\
&\quad\quad\times\big(1+\beta V(x,y)+\beta V(x',y')\big) +2\beta C_V^* f(r)\\
&=\big(f'_{-}(r)  ((\aa_0+1)b\gamma^{-1}+L)|q|+2 \gamma^{2} f''(r)\big)h_\vv(|q|)^2\\
&\quad\quad\times\big(1+\beta V(x,y)+\beta V(x',y')\big) \\
&\quad+f'_{-}(r)  \big((\aa_0+1)b\gamma^{-1}+L\big)|q|(1-h_\vv(|q|)^2)\\
&\quad\quad\times\big(1+\beta V(x,y)+\beta V(x',y')\big)  +2\beta C_V^* f(r).
\end{align*}
According to the choice of $c_2$ given in \eqref{U14}, we obviously have $ ((\aa_0+1)b\gamma^{-1}+L)\ell_0=\frac{1}{2} c_2  \gamma^{2}$.
Then,  by invoking $f'(r)=c_1+c_2\e^{-c_2r}\le 2c_2\e^{-c_2r}, 0\le r\le \ell_0 $ (recalling $c_1=c_2\e^{-c_2\ell_0}$) and
 $f''(r)=-c_2^2\e^{-c_2r}$,
we find that
\begin{align*}
 \big(f'_{-}(r)  ((\aa_0+1)b\gamma^{-1}+L)|q|+2 \gamma^{2} f''(r)\big)h_\vv(|q|)^2
 &\le -\frac{1}{\ell_0}c_1c_2     \gamma^{2} |q|  h_\vv(|q|)^2\\
 &=-\frac{1}{\ell_0}c_1c_2       \gamma^{2} |q|  +\frac{1}{\ell_0}c_1c_2      \gamma^{2} |q|(1-h_\vv(|q|)^2)\\
 &\le -\frac{1}{\ell_0}c_1c_2     \gamma^{2} |q| +\frac{4\vv}{\ell_0}c_1c_2       \gamma^{2} \\
 &\le  -\frac{c_1c_2  \gamma^{2} f(r) }{\ell_0(\aa_0\kk_0+1)(c_1+c_2)}    +\frac{4}{\ell_0}c_1c_2     \gamma^{2}\vv,
\end{align*}
where in the second inequality we utilize $|q|(1-h_\vv(|q|)^2)\le 4\vv$, and in the last display we exploit $r\le  (1+\aa_0\kk_0)|q|$
and $f(r)\le (c_1+c_2)r$. Subsequently, for $C_*>0$ given in \eqref{U17}, we obtain that
 \begin{align*}
\Theta^\vv_1((x,y),(x',y'))
&\le -\frac{c_1c_2  \gamma^{2}   }{\ell_0(\aa_0\kk_0+1)(c_1+c_2)} \rho_{V,g,\bb}(( x,y),(  x',y'))   +2\beta C_V^* f(r)\\
&\quad+C_*\vv\big(1+\beta V(x,y)+\beta V(x',y')\big) \\
&\le  -\frac{c_1c_2  \gamma^{2}   }{2\ell_0(\aa_0\kk_0+1)(c_1+c_2)} \rho_{V,g,\bb}(( x,y),(  x',y'))\\
 &\quad+C_*\vv\big(1+\beta V(x,y)+\beta V(x',y')\big),
\end{align*}
where in the second inequality we take advantage of
\begin{align*}
2\beta C_V^*\le \frac{c_1c_2  \gamma^{2}   }{2\ell_0(\aa_0\kk_0+1)(c_1+c_2)}
\end{align*}
by taking the definition of $\beta$ defined in \eqref{U18} into account.

Based on the preceding analysis, the assertion \eqref{U11} can be attainable.
\end{proof}

Next, we move forward to carry out the proof of Theorem \ref{thm2}.

\begin{proof}[Proof of Theorem \ref{thm2}]
Below, we stipulate $\delta\in(0,\delta^*]$ so that \eqref{F10} is available. 
In terms of the definition of $\beta$ given in \eqref{U18}, it is ready to see that
\begin{align*}
 2L^*_V(1+c_2/{c_1})\beta^{1-\eta}
\big(1\vee (({1+\gamma})/{\aa_0})\big)\le \frac{1}{2}\lambda_*.
\end{align*}
Whereafter, by combining  Lemma \ref{lem8} with Lemma \ref{lem9}, we infer that
\begin{align*}
 \d \rho_{V,g,\bb}\big((  X^{\dd,\vv}_t,  Y^{\dd,\vv}_t),(   X^{(\dd,\vv)}_t,    Y^{(\dd,\vv)}_t)\big)
&\le -\lambda^*\rho_{V,g,\bb}((  X^{\dd,\vv}_t,  Y^{\dd,\vv}_t),( X^{(\dd,\vv)}_t,    Y^{(\dd,\vv)}_t))\,\d t+R^{ \dd,\vv}_t\,\d t\\
&\quad+C_*\vv (1+\beta  V( X^{\dd,\vv}_t, Y^{\dd,\vv}_t)+\beta V(   X^{(\dd,\vv)}_t,  Y^{(\dd,\vv)}_t) )\,\d t\\
&\quad+\d M^{\dd,\vv}_t,
\end{align*}
in which
 \begin{align*}
 \lambda^*:=\ff12\ll_*\wedge\ff {2	\bb C_V^*}{1+4\bb C_V^*/\ll^*_V},
\end{align*}
 and
 $$R^{\dd,\vv}_t:=\beta g(r^{\dd,\vv}_t)\Phi^{\dd,\vv}_t +\big( 1+\bb V(  X^{ \dd,\vv }_t,  Y^{ \dd,\vv }_t)+\bb V(    X^{(\dd,\vv)}_t,   Y^{(\dd,\vv)}_t)\big)g'_{-}( r^{\dd,\vv}_t)\Psi^{\dd,\vv}_t
 $$
 with $\Phi^{\dd,\vv}$ and $\Psi^{\dd,\vv}$ being given in  \eqref{F11} and \eqref{F12}, respectively. 
Thus, recalling $$\big((  X^{\dd,\vv}_0,  Y^{\dd,\vv}_0),( X^{(\dd,\vv)}_0,   Y^{(\dd,\vv)}_0)\big)=\big((  X_0^\nu,  Y_0^\nu),(  X^{(\dd),\mu}_0,  Y^{(\dd),\mu}_0)\big)$$  and applying Gronwall's inequality
enables us to derive that
\begin{align*}
 &\E\rho_{V,g,\bb}\big((   X^{\dd,\vv}_t,   Y^{\dd,\vv}_t),(   X^{(\dd,\vv)}_t, Y^{(\dd,\vv)}_t)\big)\\
&\le \e^{-\lambda^*t}\E\rho_{V,g,\bb}\big((  X_0^\nu,  Y_0^\nu),(  X^{(\dd),\mu}_0,  Y^{(\dd),\mu}_0)\big)\\
&\quad+\int_0^t\e^{-\lambda^*(t-s)}\big(C_*\vv \big(1+\beta  \E V(  X^{\dd,\vv}_s,   Y^{\dd,\vv}_s)+\beta \E V(    X^{(\dd,\vv)}_s,   Y^{(\dd,\vv)}_s)\big )+\E R^{\dd,\vv}_s\big)\,\d s.
\end{align*}
Next,
due to  $\beta\in(0,1]$ and $f'(u)\le c_1+c_2,u\ge0,$ along with  $({\bf B}_1)$ and \eqref{U19},
there is a constant  $C_1 >0$ such that
\begin{align*}
|R^{\dd,\vv}_t|&\le \beta f(\ell_0)|\Phi^{\dd,\vv}_t|+(c_1+c_2)\big( 1+ V(X^{\dd,\vv}_t,  Y^{\dd,\vv}_t)+ V(    X^{(\dd,\vv)}_t,    Y^{(\dd,\vv)}_t)\big) |\Psi^{\dd,\vv}_t|\\
&\le C_1\big(1+V(   X^{\dd,\vv}_t,   Y^{\dd,\vv}_t)+ V(   X^{(\dd,\vv)}_t,   Y^{(\dd,\vv)}_t) \big) \big(  |  X^{(\dd,\vv)}_t- X^{(\dd,\vv)}_{t_\dd} |+  | Y^{(\dd,\vv)}_t- Y^{(\dd,\vv)}_{t_\dd} |\big).
\end{align*}
This, together with the facts that $\rho_V$ is equivalent to $\rho_{V,g,\bb}$ and $(( X_t^{\dd,\vv},  Y_t^{\dd,\vv}), (X^{(\dd,\vv)}_t,Y^{(\dd,\vv)}_t))_{t\ge0}$
 is a coupling process of $(X_t^\nu,Y_t^\nu)_{t\ge0}$ and $(X_t^{(\delta),\mu},Y_t^{(\delta),\mu})_{t\ge0}$, implies that for some constant $C_2>0,$
\begin{equation}\label{U20}
\begin{split}
 &\E\rho_{V }\big((   X^{\dd,\vv}_t,   Y^{\dd,\vv}_t),(   X^{(\dd,\vv)}_t, Y^{(\dd,\vv)}_t)\big)\\
&\le C_2\e^{-\lambda^*t}\E\rho_{V}\big((  X_0^\nu,  Y_0^\nu),(  X^{(\dd),\mu}_0,  Y^{(\dd),\mu}_0)\big)\\
&\quad+C_2\vv \int_0^t\e^{-\lambda^*(t-s)}\big(1+ \E V(  X_s^\nu,  Y_s^\nu)+ \E V(  X^{(\dd),\mu}_s,  Y^{(\dd),\mu}_s) ) \big)\,\d s\\
&\quad+C_2\int_0^t\e^{-\lambda^*(t-s)}\E\Big(\big(1+V(   X^{\dd,\vv}_s,   Y^{\dd,\vv}_s)+ V(   X^{(\dd,\vv)}_s,   Y^{(\dd,\vv)}_s) \big)\\
&\qquad\qquad\qquad\qquad\qquad\times\big(  |  X^{(\dd,\vv)}_s- X^{(\dd,\vv)}_{s_\dd} |+  | Y^{(\dd,\vv)}_s- Y^{(\dd,\vv)}_{s_\dd} |\big)\Big)\,\d s.
\end{split}
\end{equation}
Under $({\bf B}_2)$,  it is more or less standard that  $\sup_{t\ge0}\E V(  X_t^\nu,  Y_t^\nu)<\infty$. Whence, by choosing the initial data
$((X_0^\nu,Y_0^\nu),(X^{(\dd),\mu}_0,Y^{(\dd),\mu}_0))$ such that
\begin{align*}
	\mathcal W_{\rho_V}(\mu,\nu)=\E\rho_{V}\big((  X_0^\nu,  Y_0^\nu),(  X^{(\dd),\mu}_0,  Y^{(\dd),\mu}_0)\big),
\end{align*}
taking advantage of \eqref{F10}, 
and approaching $\vv\to0$ in \eqref{U20} yields    \eqref{W12}. 
\end{proof}

\subsection{Proof of Theorem   \ref{prolem}}
The establishment of Theorem \ref{thm2} is based on the hypothesis that $\E V(X_{t}^{(\dd),\nu},Y_t^{(\dd),\nu})<\8$ for each $t>0,$ where $V$
is the associated Lyapunov function. So, in order to apply Theorem \ref{thm2} to the proof of Theorem   \ref{prolem}, it is primary to show that $(X_t^{(\delta)},Y_t^{(\delta)})_{t\ge0}$ governed by \eqref{E2-2}
possesses  a uniform-in-time $p$-th moment, which is also important to tackle
the remainder term in \eqref{W12}. For this, we prepare for the following lemma.

\begin{lemma}\label{llem}
Assume $({\bf A}_3)$. Then, for any $p>0$, there exist constants $C_p^*,C_p^{**},\ll_p,	\dd^*_p>0$
 such that for all $\dd\in(0,\dd^*_p),$
\begin{align}\label{lle1}
\E\big( (|X^{(\delta)}_t|^p+ |Y^{(\delta)}_t|^p)\big|\mathscr F_0\big)\le C^*_p\e^{-\ll_pt}(|X^{(\delta)}_0|^p+|Y^{(\delta)}_0|^p)+C^{**}_p(1+ d^{\frac{1}{2}p}).
\end{align}
\end{lemma}

\begin{proof}
By H\"older's inequality, it suffices to show \eqref{lle1} for any integer $p\ge3$.
Below, we fix  $ \ll\in(0,2\ss{\lambda_1}\wedge\frac{1}{4}) $, $\delta\in(0,1]$ and $p\ge3$.
Recall that
\begin{align*}
V_\lambda(x,y) =1+U(x)-\min_{x\in\R^d}U(x)+\ff14\big(|x+y|^2+|y|^2-\ll|x|^2\big), \quad x,y\in \R^d.
\end{align*}
Due to $ \ll\in(0,2\ss{\lambda_1}\wedge\frac{1}{4}), $  one has $V_\lambda(x,y)\ge1$ for any $x,y\in\R^d.$
It is easy to see that for all $x,y\in\R^d,$
\begin{align}\label{U24}
\Big(\frac{1}{4}-\lambda\Big)|x|^2+\frac{2}{3}|y|^2\le |x+y|^2+|y|^2-\ll|x|^2\le\Big(\frac{7}{4}-\lambda\Big)|x|^2+\frac{10}{3}|y|^2.
\end{align}
This, together with \eqref{E2-2}, implies that for all $x,y\in\R^d,$
\begin{equation}\label{U23}
\begin{split}
\<x,\nn U (x)\>+  |y|^2 + \ll \<x,y\>&\ge \ll_1 |x|^2+\ll_2 U(x)-C^\star+  |y|^2 + \ll \<x,y\>\\
%&\ge \frac{1}{8} (4\ll_1- \lambda^2 ) |x|^2+\ll_2 U(x)-C^\star+ \frac{4\lambda_1-\lambda^2}{\lambda^2+4\lambda_1}  |y|^2\\
%&\ge(\kk_1\wedge\kk_2)\bigg(\Big(\frac{7}{4}-\lambda\Big)  |x|^2+ \frac{10}{3}  |y|^2\bigg)+ \ll_2\Big( 1+U(x)-\min_{x\in\R^d}U(x)\Big)\\
%&\quad -\lambda_2+\ll_2\min_{x\in\R^d}U(x)-C^\star\\
%&\ge (\kk_1\wedge\kk_2)\big(|x+y|^2+|y|^2-\ll|x|^2\big)+ \ll_2\Big( 1+U(x)-\min_{x\in\R^d}U(x)\Big)\\
 &\ge\lambda_*V_\lambda(x,y)-\lambda_2+\ll_2\min_{x\in\R^d}U(x)-C^\star,
\end{split}
\end{equation}
where
\begin{align*}
\lambda_*:=(4\kk)\wedge\ll_2\quad \mbox{ with } \quad  \kk:=\frac{4\ll_1- \lambda^2 }{8(\frac{7}{4}-\lambda ) }\wedge\frac{3(4\lambda_1-\lambda^2)}{10(\lambda^2+4\lambda_1)}.
\end{align*}

By applying It\^o's formula, we obtain from \eqref{U23} that
\begin{equation}\label{c2}
\begin{aligned}
	\d V_\lambda(X^{(\delta)}_t,Y^{(\delta)}_t)
	&=- \frac{1}{2}\big(  \<X^{(\delta)}_t,\nn U (X^{(\delta)}_t)\>+  |Y^{(\delta)}_t|^2 + \ll \<X^{(\delta)}_t,Y^{(\delta)}_t\>- d \big)\,\d t
\\
&\quad+   \Theta_t^\dd\,\d t+\ff12 \< X^{(\delta)}_t+2Y^{(\delta)}_t, \d W_t \>\\
&\le -\frac{1}{2}\lambda_*V_\lambda(X^{(\delta)}_t,Y^{(\delta)}_t)\,\d t+\Theta_t^\dd\,\d t+ C_{ 1}\,\d t+\ff12 \< X^{(\delta)}_t+2Y^{(\delta)}_t, \d W_t \>,
	\end{aligned}
	\end{equation}
where
\begin{align*}
 \Theta_t^\dd:= \ff12\big\< X^{(\delta)}_t+2Y^{(\delta)}_t,\nn U(X^{(\delta)}_t)-\nn U(X^{(\delta)}_{t_\dd})\big\> -\ff12 \<Y^{(\delta)}_t+\ll X^{(\delta)}_t-2\nn U(X^{(\delta)}_t), Y^{(\delta)}_{t_\dd}-Y^{(\delta)}_t\>,
\end{align*}
 and
\begin{align*}
C_{ 1}:=\frac{1}{2}\Big(d+\lambda_2+C^\star+\ll_2\Big|\min_{x\in\R^d}U(x)\Big|\Big).
\end{align*}
Note that
\begin{align}\label{U26}
X^{(\delta)}_t-X^{(\delta)}_{t_\dd}=Y^{(\delta)}_{t_\dd}(t-t_\dd) ~ \mbox{and} ~  Y^{(\delta)}_t-Y^{(\delta)}_{t_\dd}=-(\nn U(X^{(\delta)}_{t_\dd})+Y^{(\delta)}_{t_\dd})(t-t_\dd)+W_t^\delta(t-t_\dd)^{\frac{1}{2}}.
\end{align}
where $W_t^\delta:=(t-t_\delta)^{-\frac{1}{2}}(W_t-W_{t_\dd})\sim N({\bf 0},I_d).$
\eqref{U26},  along with \eqref{U-} and \eqref{U24},  implies that for some constant $C_{ 2}>0 $,
\begin{align*}
 \Theta_t^\dd&\le\frac{1}{2} \big((\lambda+2L_U)|X^{(\delta)}_t|+2|\nn U({\bf0})|+|Y^{(\delta)}_t| \big)\\
   &\quad\times\big( \big(L_U|X^{(\delta)}_{t_\dd}|+|\nn U({\bf0})|   +|Y^{(\delta)}_{t_\dd}|\big)(t-t_\dd)+|W_t^\delta|(t-t_\dd)^{\frac{1}{2}}\big)\\
 &\quad+\ff12 L_U(| X^{(\delta)}_t|+2|Y^{(\delta)}_t|) |Y^{(\delta)}_{t_\dd}|(t-t_\dd) \\
 &\le \frac{1}{16}\lambda_*\bigg(\Big(\frac{1}{4}-\lambda\Big)| X^{(\delta)}_t|^2+\frac{2}{3}|Y^{(\delta)}_t|^2\bigg) +C_{ 2}\big(1+( |X^{(\delta)}_{t_\dd}|^2+|Y^{(\delta)}_{t_\dd}|^2 )\delta^2+|W_t^\delta|^2\delta    \big)\\
 &\le   \frac{1}{4}\lambda_*V_\lambda(X^{(\delta)}_t,Y^{(\delta)}_t) +C_{*,2}\big(1 +( |X^{(\delta)}_{t_\dd}|^2+|Y^{(\delta)}_{t_\dd}|^2 )\delta^2 +|W_t^\delta|^2\delta  \big).
\end{align*}
Plugging the previous estimate  back into \eqref{c2} followed by applying Gronwall's inequality and making use of \eqref{U24}  yields that for any integer $n\ge0,$
\begin{equation*}
\begin{aligned}
  \e^{\frac{1}{4}\lambda_*(n+1)\delta } V_\lambda(X^{(\delta)}_{(n+1)\dd},Y^{(\delta)}_{(n+1)\dd})
 &\le    \e^{\frac{1}{4}\lambda_* n \delta }V_\lambda(X^{(\delta)}_{n\dd},Y^{(\delta)}_{n\dd})+N_n^\dd+M_n^\dd \\
  &\quad+ C_{ 2}( |X^{(\delta)}_{n\dd}|^2+|Y^{(\delta)}_{n\dd}|^2 )\delta^2 \int_{n\delta}^{(n+1)\delta}  \e^{\frac{1}{4}\lambda_*s } \,\d s,
	\end{aligned}
	\end{equation*}
in which
\begin{align*}
N_n^\dd:=C_{ 2}\int_{n\delta}^{(n+1)\delta}\e^{\frac{1}{4}\lambda_*s}(1+|W_s^\delta|^2\delta  )\,\d s \quad \mbox{ and } \quad  M_n^\dd:=\ff12\int_{n\delta}^{(n+1)\delta}  \e^{\frac{1}{4}\lambda_*s }\< X^{(\delta)}_s+2Y^{(\delta)}_s, \d W_s\>.
\end{align*}
Subsequently, by invoking \eqref{U24} again and using the fundamental inequality: $\e^{r}-1\le r\e^r,r\ge0,$ there exists a constant $C_{ 3}>0$ such that
\begin{align*}
 \e^{\frac{1}{4}\lambda_*(n+1)\delta } V_\lambda(X^{(\delta)}_{(n+1)\dd},Y^{(\delta)}_{(n+1)\dd})
 &\le    \e^{\frac{1}{4}\lambda_* n \delta }(1+C_{ 3}\delta^3)V_\lambda(X^{(\delta)}_{n\dd},Y^{(\delta)}_{n\dd})+N_n^\delta+M_{n}^\delta.
\end{align*}
This further implies that
\begin{equation*}
\begin{aligned}
	&\e^{\ff14 {\ll_* p(n+1)\dd} }  V_\ll(X^{(\delta)}_{(n+1)\dd},Y^{(\delta)}_{(n+1)\dd})^p
-\e^{\ff14 {\ll_* p n\dd} }(1+C_{ 3}\dd^3) ^p V_\ll(X^{(\delta)}_{n\dd},Y^{(\delta)}_{n\dd})^p\\
&\le p\e^{\ff14 {\ll_*(p-1) n\dd} }(1+C_{ 3}\dd^3)^{p-1} V_\ll(X^{(\delta)}_{n\dd},Y^{(\delta)}_{n\dd})^{p-1}( N^\delta_n+M^\delta_n)\\
	&\quad +\sum_{k=2}^pC_p^k\e^{\ff14 {\ll_* (p-k) n\dd} }(1+C_{ 3}\dd^3) ^{p-k} V_\ll(X^{(\delta)}_{n\dd},Y^{(\delta)}_{n\dd})^{p-k}   ( N^\delta_n+M^\delta_n)^k =:\Xi^\delta_{p,n}+\bar \Xi^\delta_{p,n}.
\end{aligned}
\end{equation*}

On the one hand,
by the aid of $\E( M_n^\delta|\mathscr F_{n\delta})=0$ and $\E( |W_s^\delta|^2|\mathscr F_{n\delta})=d$, besides  the inequality $\e^{r}-1\le r\e^r,r\ge0,$
it follows from Young's inequality that for some constant $C_{ 4}>0,$
\begin{equation}
\begin{split}
\E( \Xi^\delta_{p,n}|\mathscr F_{n\delta})&\le C_{ 4}p\e^{\ff14 {\ll_* p  n\dd} }(1+C_{ 3}\dd^3)^{p-1} V_\ll(X^{(\delta)}_{n\dd},Y^{(\delta)}_{n\dd})^{p-1}(1+d)\delta\\
&\le \frac{1}{16}p\lambda_*\delta\e^{\ff14 {\ll_* p  n\dd} }(1+C_{ 3}\dd^3)^{p-1}V_\ll(X^{(\delta)}_{n\dd},Y^{(\delta)}_{n\dd})^{p} +C_{ 4}\e^{\ff14 {\ll_* p  n\dd} }(1+d^p)\delta.
\end{split}
\end{equation}
On the other hand, via H\"older's inequality and the moment estimate for stochastic integrals (see e.g. \cite[Theorem 7.1, p.39]{Mao08}), there are constants $C_{ 5},C_{ 6}>0$  such that  for  $k=2,\cdots,p,$
\begin{align*}
	\E\big( |N_n |^k + |M_n |^k\big|\mathscr F_{n\dd}\big)&\le C_{ 2}^k\dd^{k-1}\int_{n\delta}^{(n+1)\delta}\e^{\frac{1}{4}k\lambda_*s}\big(1+\E |W_s^\delta|^{2k}\delta^k  \big)\,\d s\\
&\quad+C_{ 5}\delta^{\frac{1}{2} k-1} \int_{n\dd}^{(n+1)\dd}\e^{\frac{1}{4}k\lambda_*s }  \big( | X^{(\delta)}_{n\delta}|^k+ | Y^{(\delta)}_{n\delta}|^k \big)  \d  s\\
&\quad+C_{ 5}\delta^{\frac{1}{2} k-1} \int_{n\dd}^{(n+1)\dd}\e^{\frac{1}{4}k\lambda_*s } \E\big( \big( |X^{(\delta)}_s-X^{(\delta)}_{n\delta}|^k+ |Y^{(\delta)}_s-Y^{(\delta)}_{n\delta}|^k\big)\big|\mathscr F_{n\dd}\big)  \d  s \\
	&\le C_{ 6} \e^{\frac{1}{4}k\lambda_*n\delta}\big(1+V_\ll(X^{(\delta)}_{n\dd},Y^{(\delta)}_{n\dd})^{\frac{1}{2}k}\big)(1+d^k )\delta^{\frac{1}{2} k },
\end{align*}
where in the second inequality we exploit  Young's inequality, \eqref{U24}, \eqref{U26} as well as
the fact that $\E |W_s^\delta|^{2k}\le c_*d^k$ for some $c_*>0$. The previous estimate, along with Young's inequality once more,  thus
 implies that for some constants $C_7,C_8>0,$
\begin{align*}
\E\big(\bar \Xi_{p,n\dd}\big|\mathscr F_{n\delta}\big)&\le C_{ 7}\e^{\ff14 {\ll_* pn\dd} }\sum_{k=2}^p(1+d^k )C_p^k(1+C_{ 3}\dd^3) ^{p-k}\delta^{\frac{1}{2} k } V_\ll(X^{(\delta)}_{n\dd},Y^{(\delta)}_{n\dd})^{p-k} \\
&\qquad\qquad\qquad\times\big(1+V_\ll(X^{(\delta)}_{n\dd},Y^{(\delta)}_{n\dd})^{\frac{1}{2}k}\big)\\
&\le \frac{1}{16}p\lambda_*\delta\e^{\ff14 {\ll_* p  n\dd} }(1+C_{ 3}\dd^3)^{p-1}V_\ll(X^{(\delta)}_{n\dd},Y^{(\delta)}_{n\dd})^{p} +C_{ 8}\e^{\ff14 {\ll_* p  n\dd} }(1+d^p)\delta .
\end{align*}
So we find that for some constant $C_9>0,$
\begin{align*}
	 \e^{\ff14 {\ll_* p \dd} } \E\big( V_\ll(X^{(\delta)}_{(n+1)\dd},Y^{(\delta)}_{(n+1)\dd})^p\big|\mathscr F_{n\delta}\big)
& \le   (1+C_{ 3}\dd^3)^{p-1}(1+C_{ 3}\dd^3+p\lambda_*\delta /8)  V_\ll(X^{(\delta)}_{n\dd},Y^{(\delta)}_{n\dd})^p  \\
 &\quad+C_{ 9} (1+d^p)\delta .
\end{align*}
Note that there exists $\delta_p^*\in(0,1]$ such  that for any $\delta\in(0,\delta_p^*]$,
$$(1+C_{ 3}\dd^3)^{p-1}(1+C_{ 3}\dd^3+p\lambda_*\delta /8)\le  1+3p\lambda_*\delta /{16}.$$
Consequently, we have that  for all $\delta\in(0,\delta_p^*]$,
\begin{equation*}
 \e^{\ff14 {\ll_* p \dd} } \E\big( V_\ll(X^{(\delta)}_{(n+1)\dd},Y^{(\delta)}_{(n+1)\dd})^p\big|\mathscr F_{n\delta}\big)
\le   (1+3p\lambda_*\delta /{16} ) V_\ll(X^{(\delta)}_{n\dd},Y^{(\delta)}_{n\dd})^p  +C_{ 9} (1+d^p)\delta .
\end{equation*}
With this estimate at hand, an iteration argument enables us to derive that for all $\delta\in(0,\delta_p^*]$,
\begin{equation}\label{U28}
\begin{split}
  \E\big( V_\ll(X^{(\delta)}_{ n \dd},Y^{(\delta)}_{  n \dd})^p\big|\mathscr F_{0}\big)
&\le  \e^{-\ff14 { \ll_* p n\dd} } (1+3p\lambda_*\delta /{16} )^n \Big(V_\ll(X^{(\delta)}_0,Y^{(\delta)}_0)^p  +\frac{16C_{ 9} (1+d^p)}{3p\lambda_*}\Big)\\
&\le  \e^{-\frac{1}{16} { \ll_* p n\dd} } \Big(V_\ll(X^{(\delta)}_0,Y^{(\delta)}_0)^p  +\frac{16C_{ 9} (1+d^p)}{3p\lambda_*}\Big),
\end{split}
\end{equation}
where in the second inequality we employ the basic inequality: $ 1+r\le \e^r, r\ge0.$  At length,
 the assertion \eqref{lle1} is achievable  by combining  \eqref{U28} with \eqref{U26}, and noting that $V_\lambda(x,y)$ is comparable with $1+|x|^2+|y|^2$.
\end{proof}

In the end of this paper, we proceed to finish the proof of Theorem \ref{prolem}.

\begin{proof}[Proof of Theorem \ref{prolem}]
Below, we stipulate $\delta\in(0,\delta^*_2]$ so   Lemma \ref{llem} is applicable, and  
for notational simplicity,  set for any $t\ge0 $ and $\vv>0,$
	\begin{align*}
R^{\dd,\vv}_t:=\big(V (X_t^{\dd,\vv},Y_t^{\dd,\vv})+ V (X^{(\dd,\vv)}_t,Y^{(\dd,\vv)}_t)\big) \big( | X^{(\dd,\vv)}_t- X^{(\dd,\vv)}_{t_\dd} |+ | Y^{(\dd,\vv)}_{t}- Y^{(\dd,\vv)}_{t_\dd } |\big),
\end{align*}
in which $V=V_\lambda$, defined in \eqref{F7}.

In terms of  Theorem \ref{thm2}, to obtain the assertion \eqref{U29},
it suffices to show that there exists a constant $C_0>0,$  such that
\begin{align}\label{U30}
\E  R^{\dd,\vv}_t \le C_0\big(1+\E|X_0^\nu|^3+\E|Y_0^\nu|^3+\E|X_0^{(\dd),\mu }|^3+\E|Y_0^{(\dd),\mu }|^3\big)( 1+d^{\ff32})\dd^{\ff12}.
\end{align}	
Indeed, by H\"older's inequality, it follows   that 	
\begin{align*}
\E (R^{\dd,\vv}_t|\mathscr F_0)&\le\Big( \big(\E \big(V (X_t^{(\dd,\vv)},Y_t^{(\dd,\vv)}\big)^2|\mathscr F_0)\big)^{\ff12}+\big(\E \big(V (X_t^{ \dd,\vv} ,Y_t^{\dd,\vv} )^2|\mathscr F_0\big)\big)^{\ff12}\Big)\\
&\qquad\times\Big(\big(\E\big(\big| X^{(\dd,\vv)}_t- X^{(\dd,\vv)}_{t_\dd}\big|^2\big|\mathscr F_0\big)\big)^{\ff12} +\big(\E\big(\big| Y^{(\dd,\vv)}_t- Y^{(\dd,\vv)}_{t_\dd}\big|^2\big|\mathscr F_0\big)\big)^{\ff12}\Big) .
\end{align*}
Notice from \eqref{U23} that there exists a constant $C_1>0$ such that
 $x,y\in\R^d,$
\begin{equation}\label{U34}
\begin{split}
(\mathscr L  V^2)(x,y)& =-  V(x,y) \big(  \<x,\nn U (x)\>+  |y|^2 + \ll \<x,y\>- d - V(x,y)^{-1}|\nn_2V_\lambda(x,y)|^2 \big)\\
& \le-\frac{1}{2} \lambda_*V (x,y)^2+C_1(1+d^2).
\end{split}
\end{equation}
This, together with the fact that $(X_t^{\dd,\vv},Y_t^{\dd,\vv})_{t\ge0}$ and $(X_t,Y_t)_{t\ge0}$ share  the same infinitesimal generator,  implies that 
there exists a constant $C_2>0$ such that
\begin{align}\label{F5}
	\E \big(V  (X_t^{\dd,\vv},Y_t^{\dd,\vv})^2\big|\mathscr F_0\big)\le C_2\big(1+  V (X_0^\nu,Y_0^\nu)^2 +d^{ 2}\big).
\end{align}
Furthermore, it follows from \eqref{EQ3} that 
\begin{equation}\label{F6}
\begin{split}
 |X^{(\dd,\vv)}_t- X^{(\dd,\vv)}_{t_\dd}|&\le|aX_{t_\delta}^{(\delta,\vv)}+bY_{t_\delta}^{(\delta,\vv)}|(t-t_\delta)\\
  |Y^{(\dd,\vv)}_t- Y^{(\dd,\vv)}_{t_\dd}|&\le|U(  X^{(\dd,\vv)}_{t_\dd}, Y^{(\dd,\vv)}_{t_
 \dd})|(t-t_\delta)+|\bar W^\delta_t|(t-t_\delta)^{\frac{1}{2}}, 
\end{split}
\end{equation}
where  $\bar W^\delta_t:=(\bar W_t-\bar W_{t_\delta})(t-t_\delta)^{-\frac{1}{2}}\in N({\bf 0}, I_d)$ with 
\begin{align*}
\bar W_t:=\int_0^th_\vv(|Q^{\dd,\vv}_s|)\Pi(Q^{\dd,\vv}_s) \d \hat W_s+\int_0^t h^*_\vv(|Q^{\dd,\vv}_s|)\,\d \tt W_s,\quad t\ge0. 
\end{align*}
In retrospect,  $( X^{(\dd,\vv)}_t,  Y^{(\dd,\vv)}_t))_{t\ge0}$  and $(X_t^{(\delta),\mu},Y_t^{(\delta),\mu})_{t\ge0}$ 
are distributed identically. Thus,  \eqref{F5} and \eqref{F6},  
in addition to Lemma \ref{llem}, \eqref{U26}, as well as the fact that $V(x,y)$ is comparable with $1+|x|^2+|y|^2$,
 give  that for some constant $C_3>0,$
\begin{equation}\label{F8}
\begin{split}
\E (R^{\dd,\vv}_t|\mathscr F_0)&\le\big( 1+ |X^{(\dd),\mu}_0|^2+|Y^{(\dd),\mu}_0|^2+|X^\nu_0|^2+|Y^\nu_0|^2 + d  \big)\\
&\quad\times\big( 1+ |X^{(\dd),\mu}_0| +|Y^{(\dd),\mu}_0| +|X^\nu_0| +|Y^\nu_0| + d^{\frac{1}{2}}  \big)\delta^{\frac{1}{2}}.
\end{split}
\end{equation}
 This subsequently implies \eqref{U30} by applying Young's inequality followed by taking expectations.

From \eqref{U34}, it is ready to see that  $\pi_\8\in \mathscr P_3(\R^{2d})$.
Whereafter, \eqref{U33} holds true by noting  $\mathbb W_1\le\mathcal W_{\rho_{ V }}$ and applying
 \eqref{U29} with $\nu=\pi_\8 $ therein.
\end{proof}

%\begin{supplement}
%\stitle{Title of Supplement A.}
%\sdescription{Short description of Supplement A.}
%\end{supplement}
%\begin{supplement}
%\stitle{Title of Supplement B.}
%\sdescription{Short description of Supplement B.}
%\end{supplement}

%%%%%%%%%%%%%%%%%%%%%%%%%%%%%%%%%%%%%%%%%%%%%%%%%%%%%%%%%%%%%%%%%%%
%%                                                               %%
%% Use the two commands below for producing your bibliography    %%
%% with bibtex, then comment again the commands and include the  %%
%% content of the .bbl file in this file below the commands.     %%
%%                                                               %%
%%%%%%%%%%%%%%%%%%%%%%%%%%%%%%%%%%%%%%%%%%%%%%%%%%%%%%%%%%%%%%%%%%%

%\bibliographystyle{amsplain}
%\bibliography{yourbibfilename}

% add below the content of your .bbl file produced by bibtex.

\end{document}